\documentclass[review,onefignum,onetabnum]{siamart190516}

\usepackage{mathtools,float,diagbox}
\usepackage{lipsum}
\usepackage{amsfonts,physics}
\usepackage{graphicx,bm}
\usepackage{epstopdf}

\newcommand{\Og}{\Omega}

\newcommand{\Dt}{\Delta}
\newcommand{\be}{\begin{equation}}
\newcommand{\ee}{\end{equation}}
\newcommand{\ba}{\begin{array}}
\newcommand{\ea}{\end{array}}
\newcommand{\bea}{\begin{eqnarray}}
\newcommand{\eea}{\end{eqnarray}}
\newcommand{\beas}{\begin{eqnarray*}}
\newcommand{\eeas}{\end{eqnarray*}}
\newcommand{\cc}{\color{blue}}
\usepackage{graphics,color,cancel}
\usepackage{amsmath,amscd,amssymb,bm}
\usepackage{mathrsfs,cite}
\usepackage{fancyhdr,fancybox}
\usepackage{listings}
\usepackage{url} 
\usepackage{enumerate,epsfig}

\usepackage{array}
\usepackage{multirow}
\usepackage{subfig}
\ifpdf
  \DeclareGraphicsExtensions{.eps,.pdf,.png,.jpg}
\else
  \DeclareGraphicsExtensions{.eps}
\fi

\newsiamremark{remark}{Remark}
\newsiamremark{hypothesis}{Hypothesis}
\crefname{hypothesis}{Hypothesis}{Hypotheses}
\newsiamthm{claim}{Claim}

\headers{Numerical Analysis of a Corrected S. Model}{F. Siddiqua, and X. Xie}

\title{Numerical Analysis of a Corrected Smagorinsky Model \thanks{
The research was partially supported by NSF grant DMS-2110379.}}
\author{Farjana Siddiqua\thanks{Department of Mathematics, University of Pittsburgh, Pittsburgh, PA-15260
  (fas41@pitt.edu ).}
\and Xihui Xie\thanks{Department of Mathematics, University of Pittsburgh, Pittsburgh, PA-15260 
  (xix55@pitt.edu).}
}

\usepackage{amsopn}

\ifpdf
\hypersetup{
  pdftitle={Numerical Analysis of a Corrected Smagorinsky Model },
  pdfauthor={Farjana Siddiqua and Xihui Xie}
}
\fi

\begin{document}
\nolinenumbers
\maketitle
\begin{abstract}
  The classical Smagorinsky model's solution is an approximation to a (resolved) mean velocity. Since it is an eddy viscosity model, it cannot represent a flow of energy from unresolved fluctuations to the (resolved) mean velocity. This model has recently been corrected to incorporate this flow and still be well-posed. Herein we first develop some basic properties of the corrected model. Next, we perform a complete numerical analysis of two algorithms for its approximation. They are tested and proven to be effective.
\end{abstract}



\section{Introduction}
Consider the Smagorinksy model \cite{Smagorinsky}\footnote{The mechanically correct formulation is with the $\grad^s w$ instead of $\grad w$ in the term $-\div\left((C_s\delta)^2|\grad w|\grad w\right)$  where $\grad^s$ is the symmetric part of the gradient tensor. But since the estimates are same and analyses are simpler with $\grad w$ due to Korn's inequality $\| v\|_{H^1(\Omega)}^2\leq 
    C[\|v\|_{L^2(\Omega)}^2+\|\grad^s v\|_{L^2(\Omega)}^2]$, we use $\grad w$ throughout the paper.},  with prescribed body force $f$, kinematic viscosity $\nu$ in the regular and bounded flow domain $\Omega\subset  \mathbb{R}^d \ (d=2,\ 3)$, which was later advanced independently by Ladyzhenskaya \cite{ladyzhenskaya1969mathematical,ladyzhenskaya1970modification}: 
$\grad\cdot w=0$ and
\begin{equation}\label{smag}
     w_t+ w\cdot\grad w-\nu\Dt w+\grad q-\div\left((C_s\delta)^2|\grad w|\grad w\right)=f( x).
\end{equation}
Here $( w,q)$ 
approximate an ensemble average of Navier-Stokes solutions, $(\overline{ u},\overline{p})$. 
This is an eddy viscosity model with turbulent viscosity, $\nu_T=(C_s\delta)^2|\grad w|$, where $C_s\approx 0.1$, 
Lilly \cite{lilly}, $\delta$ is a length scale (or grid scale). Like all eddy viscosity models, the Smagorinsky model represents a flow of energy from means to unresolved fluctuations ($ u'= u-\overline{ u}$, for a precise formula see \cref{def:2.14}) and has errors by not representing any intermittent energy flow from fluctuations back to means. 
Corrections have recently been made representing this flow in  Jiang and Layton \cite{jiang2016ev} and Rong, Layton and Zhao \cite{rong2019extension}.
Following their ideas, we develop a corrected model in \cref{sec:3}. We also analyze and test numerical algorithms for effective approximation of the resulting corrected model: $\div w=0$ and
\begin{equation}\label{CSM0}
 w_t-C_s^4\delta^2\mu^{-2}\Dt w_t+ w\cdot\grad w-\nu\Dt w+\grad q-\div\Big((C_s\delta)^2|\grad w|\grad w\Big)=f( x).
\end{equation}
Here $\mu$ is a constant from Kolmogorov-Prandtl relation  \cite{kolmogrov,prandtl}.

The main result of this paper is the complete numerical analysis and computational testing of effective algorithms for this 
model. 
This paper gives detailed 
numerical analyses in \cref{sec:4} and \cref{sec:5}.
This model is able to capture the phenomenon of transferring energy from fluctuation to means,
 which is tested numerically in \cref{sec:6.2}. 
 There were few attempts made for extending model that represents flow at statistical equilibrium to non-equilibrium. For instance, in a previous work by Jiang and Layton \cite{jiang2016ev}, there was an extra fitting parameter $\beta$ in the second term of \eqref{CSM0} which is needlessly complicated. In our paper, a different idea results in a simpler model with no new fitting parameters other than from the Smagorinsky model \eqref{smag}.



\subsection{Previous work}
For simulating turbulent flow, there are different approaches,
see \cite{turbulence,turbulence1,turbulence2,turbulence3,turbulence4,pakzad2017damping}. 
A summary of some recent work in eddy viscosity models of turbulence is presented in \cite{layton2020EV}.
One of the recent approaches is by adding a term of Kelvin-Voigt form to the equations for the mean-field  \cite{amrouche2020turbulent}. Smagorinsky model is a classical model. It's positive and negative features are well understood. There has been lot of work correcting negative features, for example Tommy K. Kim \cite{kim2001modified} did a different modification than ours which corrects near wall behavior. The new term in our model has similarity to the Voigt term used in Voigt/Kelvin-Voigt/Kelvin Model \cite{stieger2002rheology} for viscoelastic fluids. There has been lot of recent works on Voigt Model, see for example \cite{larios2018computational,larios2014higher,kuberry2012numerical,baranovskii2020strong}.
Recently, Rong, Layton and Zhao \cite{rong2019extension} and Berselli, Lewandowski and Nguyen \cite{rotation} all studied the extension of the Baldwin \& Lomax model \cite{baldwin1978thin} to non-equilibrium ($\frac{d}{dt}\overline{\| u'\|^2}\neq 0$, for a precise definition see \cref{equilibirum}) 
problems.
A variant of the Smagorinsky model and detailed analysis is presented in the paper \cite{ChorfiAbdelwahedBerselli+2020+1402+1419}.
Jiang and Layton \cite{jiang2016ev} derived a corrected eddy viscosity model for flow not at statistical equilibrium state.

\section{Notation and Preliminaries}

In this section, we introduce some of the notations and results used in this paper.
We denote by $\|\cdot\|$ and $(\cdot,\cdot)$ the $L^2(\Omega)$ norm and inner product, respectively. We denote the $L^p(\Omega)$ norm by $\|\cdot\|_{L^p}$.
The solution spaces $X$ for the velocity and $Q$ for the pressure are defined as:
\begin{equation*}
\begin{aligned}
    &X:=\{ v\in L^3(\Omega): \grad  v\in L^3(\Omega)\ \text{and}\  v=0\ \text{on}\ \partial\Omega\},\\
    &Q:=L^2_0(\Omega)=\{q\in L^2(\Omega): \int_\Omega q\ d x=0\},\\
    \text{and}\quad &V:=\{ v\in X: (q,\grad\cdot  v)=0,\ \forall q\in Q\}.
\end{aligned}
\end{equation*}
The space $H^{-1}(\Omega)$ denotes the dual space of bounded linear functionals defined on $H_0^1(\Omega)=\{ v\in H^1(\Omega):  v=0\ \text{on}\  \partial\Omega\}$ and this space is equipped with the norm:
\begin{equation*}
    \|f\|_{-1}=\sup_{0\neq  v\in X}\frac{(f, v)}{\|\grad  v\|}.
\end{equation*}
The finite element method for this problem involves picking finite element spaces \cite{cfdbook} $X^h\subset X$ and $Q^h\subset Q$. We assume that $(X^h, Q^h)$ satisfies the discrete inf-sup condition:
\begin{equation*}
\inf_{\lambda^h\in Q^h}\sup_{ v^h\in X^h}\frac{(\lambda^h,\grad\cdot  v^h)}{\|\lambda^h\|\|\grad  v^h\|}\geq \beta^h>0,
\end{equation*}
where $\beta^h$ is bounded away from zero uniformly in $h$.

\begin{definition}\label{trilinear0}
(Trilinear Form)
Define the skew symmetrized trilinear form $b^*:X\times{X}\times{X}\rightarrow \mathbb{R}$ as follows 
\begin{align*}
b^*( u, v, w):=\frac{1}{2}( u\cdot \grad  v, w)-\frac{1}{2}( u\cdot \grad  w, v).\\
\end{align*}
\end{definition}
\begin{lemma} \label{trilinear1} ($p.114$, Girault and Raviart \cite{GandR})
For any $u\in V$ and $v,w\in X$,
\begin{align*}
b^*( u, v, w)=( u\cdot \grad  v, w),\
\text{and}\ b^*( u, v, v)=0,\quad \forall\ u,\ v\in X.
\end{align*}
\end{lemma}

\begin{lemma}\label{trilinear ineq}
For any $ u,\  v,\  w\in X$,
\begin{equation*}
\begin{aligned}
    &\left|\int_\Omega  u\cdot\grad  v\cdot  w\ d x\right|\leq C\|\grad  u\|\|\grad  v\|\|\grad  w\|,\\&
    \left|\int_\Omega  u\cdot\grad  v\cdot  w\ d x\right|\leq C\| u\|^{1/2}\|\grad  u\|^{1/2}\|\grad  v\|\|\grad  w\|.
    \end{aligned}
\end{equation*}
\end{lemma}

\begin{lemma}(Polarization identity)
\begin{equation}\label{polar}
    ( u, v)=\frac{1}{2}\| u\|^2+\frac{1}{2}\| v\|^2-\frac{1}{2}\| u- v\|^2.
\end{equation}
\end{lemma}
\begin{lemma}(The Poincar\'{e}-Friedrichs' inequality)
There is a positive constant $C_{PF}=C_{PF}(\Omega)$ such that
\begin{equation}\label{pf ineq}
    \| u\|\leq C_{PF}\|\grad u\|,\;\;\forall u\in X.
\end{equation}
\end{lemma}

Next is a Discrete Gronwall lemma see Lemma 5.1 p.369 \cite{heywood1990}.
\begin{lemma}\label{gronwall}
Let $\Delta t,\ B, \ a_n,\  b_n, c_n, \ d_n $ for integers $n\geq 0$ be nonnegative numbers such that for $l\geq 1$, if
\begin{equation*}
a_l+\Delta t\sum_{n=0}^lb_n\leq \Delta t \sum_{n=0}^{l-1}d_na_n+\Delta t\sum_{n=0}^lc_n+B,\ \text{for} \;l\geq 0,
\end{equation*}
then for all $\Delta t>0$,
\begin{equation*}
a_l+\Delta t\sum_{n=0}^lb_n\leq \exp(\Delta t\sum_{n=0}^{l-1}d_n)\Big(\Delta t\sum_{n=0}^lc_n+B\Big),\ \text{for}\; l\geq 0.
\end{equation*}
\end{lemma}
In this paper, we will need this following well-known lemma, see, e.g., \cite{Layton2002error,layton1996,lady1991}
\begin{lemma}\label{monotonicity and LLC}
({\bf Strong Monotonicity (SM) and Local Lipschitz Continuity (LLC)})\\There exists $C_1, \ C_2>0$ such that for all $ u,\  v,\  w\in L^3(\Omega),\ \grad  u,\ \grad  v,\ \grad  w\in L^3(\Omega)$
\begin{align}
    \textbf{(SM)}&\quad  (|\grad  u|\grad  u-|\grad  w|\grad  w,\grad( u- w))\geq C_1\|\grad( u- w)\|_{L^3(\Omega)}^3,\label{strong-monotonicity} \\ 
    \textbf{(LLC)}&\quad  (|\grad  u|\grad  u-|\grad  w|\grad  w,\grad  v)\leq C_2 r \|\grad( u- w)\|_{L^3(\Omega)}\|\grad  v\|_{L^3(\Omega)},\label{LLC}
\end{align}
where $r=\max\{\|\grad  u\|_{L^3(\Omega)},\ \|\grad  w\|_{L^3(\Omega)}\}$.
\end{lemma}
\begin{proposition}(see p.173 
\cite{tool2012mathematical}) Let $W^{m,p}(\Omega)$ denote the Sobolev space, let $p\in[1,+\infty]$ and $q\in[p,p^*]$, \text{where}\  $\frac{1}{p^\star}=\frac{1}{p}-\frac{1}{d}\;\;\text{if}\;\; p<\text{dim}(\Omega)=d$. There is a $C>0$ such that
\begin{equation}\label{soblev2}
    \| u\|_{L^q(\Omega)}\leq C\| u\|_{L^p(\Omega)}^{1+d/q-d/p}\| u\|_{W^{1,p}(\Omega)}^{d/p-d/q},\quad\forall  u\in W^{1,p}(\Omega)
\end{equation}
\end{proposition}
The weak formulation of $\eqref{CSM8}$ is: Find $( w,p)\in (X,Q)$ such that
\begin{equation}\label{cont-CSM}
\begin{aligned}
    ( w_t, v) +\frac{C_s^4}{\mu^2}\delta^2 (\grad w_t,\grad  v) +( w\cdot\grad w, v)+\nu(\grad w,\grad  v)\\
    -( p,\grad\cdot  v)+\Big((C_s\delta)^2|\grad w|\grad w,\grad  v\Big)=(f, v)\   \text{for all}\  v\in X,\\
    (q,\grad\cdot w)=0\ \text{for all}\ q\in Q.\ 
\end{aligned}
\end{equation}
For the stationary Smagorinsky model, Du and Gunzburger \cite{du1990finite,lady1991} proved that the discrete solution converges to the continuous problem under minimal regularity assumptions. The existence and uniqueness of the strong solution of the Smagorinsky model \eqref{smag} on a periodic domain have been discussed \cite{malek2019weak,lee2012finite,layton2016energy}. Recently, the error estimates for Smagorinsky model have also been studied in \cite{burman2021error} and it showed that both the accuracy and the stability are enhanced for flows with high Reynolds number. The existence and uniqueness of strong solutions of the incompressible Navier-Stokes-Voigt model is studied in \cite{baranovskii2020strong}. 

Here we omit the proof for the existence of a strong solution for the new CSM Model. We assume the model has a solution in the following sense.
\begin{definition}
A solution $ w$ of the Corrected Smagorinsky Model \eqref{CSM8} is a strong solution if $ w$ satisfies the following
\begin{enumerate}
    \item $ w\in L^\infty(0,T;H^1(\Omega))\cap L^2(0,T; W^{1,3}(\Omega))\cap L^2(0,T;L^6(\Omega)),$
    \item $ w(x,t)\rightarrow  w_0(x)$ in $L^2(\Omega)$ as $t\rightarrow 0,$
    \item $ w$ satisfies the model's weak form \eqref{cont-CSM} for all $ v\in L^\infty(0,T;H^1(\Omega))\\ \cap L^2(0,T;W^{1,3}(\Omega))
    \cap L^2(0,T;L^6(\Omega))$.
\end{enumerate}
\end{definition}
\begin{remark}
Though existence of strong solutions is not yet proven for the new model, we believe it is reasonable to assume existence because it is known for the Smagorinsky Model and the extra Voigt term is linear and regularizing.
\end{remark}
\begin{definition}\label{def:2.14}
(Mean, Fluctuation and Variance) The ensembles $ u(x,t;\omega_j),$\\$\ j=1,\dots,J$\  where  $\omega$ is a a random variable, mean $\overline{ u}$ and fluctuation $ u'$ are defined as follows:
\begin{align*}
    \overline{ u}( x,t)=\frac{1}{J}\sum_{j=1}^J  u( x,t;\omega_j),\ \ \  u'( x,t,\omega_j)= u( x,t;\omega_j)-\overline{ u}( x,t).
\end{align*}
The variance of $ u$ and $\grad  u$ are, respectively,
\begin{align*}
    V( u):=\int_\Omega \overline{| u|^2}-|\overline{ u}|^2\ d x,\ \ \ V(\grad  u):=\int_\Omega \overline{|\grad  u|^2}-|\grad \overline{ u}|^2\ d x.
\end{align*}
\end{definition}

\begin{definition}\label{properties}
(Reynolds Stresses) The Reynolds stresses are 
\begin{align*}
    R( u, u):=\overline{ u}\otimes\overline{ u}-\overline{ u\otimes  u}=-\overline{ u'\otimes  u'}.
\end{align*}
\end{definition}
Ensemble averaging satisfies the following properties, e.g., \cite{k-epsilon,fluid-mechanics,random-func}
\begin{align*}
    &\overline{\overline{ u}}=\overline{ u},\quad\quad\quad\quad\quad\ \overline{ u'}=0, \quad
    \overline{\overline{ w}\cdot  v}=\overline{ w}\cdot\overline{ v}, \ \ \quad\quad 
    \overline{\overline{ w}\cdot  v'}=\overline{ w}\cdot\overline{ v'}=0,\\
    &\overline{\overline{ w}\otimes  v}=\overline{ w}\otimes\overline{ v}, \ \ \quad \overline{\overline{ w}\otimes  v'}=\overline{ w}\otimes\overline{ v'}=0, \quad
    \frac{\partial}{\partial t}\overline{ u}=\overline{\frac{\partial}{\partial t} u},\quad\quad\quad \frac{\partial}{\partial  x}\overline{ u}=\overline{\frac{\partial}{\partial  x} u}.
\end{align*}
\begin{theorem}\label{b}
Suppose that each realization is a strong solution of the NSE. The ensemble is generated by different initial data and $ u( x,0;\omega_j)\in L^2(\Omega),\ f( x,t)\in L^{\infty}(0,\infty;L^2(\Omega))$. Then the following two properties are satisfied.\\

$Property\ 1:$ (Time averaged dissipativity) 
\begin{equation*}\lim_{T\rightarrow\infty}\frac{1}{T}\int_0^T\int_{\Omega}R( u, u):\grad\bar{ u}\ d x dt=\lim_{T\rightarrow\infty}\frac{1}{T}\int_{\Omega}\nu\overline{|\grad u'|^2}\ d x dt\geq 0.
\end{equation*}
$Property\ 2:$ (Equation for the evolution of variance of fluctuations)
\begin{equation}\label{be}
\int_{\Omega}R( u, u):\grad\bar{ u}\ d x=\frac{1}{2}\frac{d}{dt}\int_{\Omega}\overline{| u'|^2}\ d x+\int_{\Omega}\nu\overline{|\grad u'|^2}\ d x.
\end{equation}
\end{theorem}
\begin{proof}
Proof of this theorem can be found in Section 2 of \cite{jiang2016ev}.
\end{proof}

\begin{remark}\label{equilibirum}
(Statistical steady state and statistical equilibrium, see \cite{jiang2016ev})  Statistical steady state is 
$\mathcal{P}/\epsilon= 1$ where
\begin{gather*}
    \epsilon=\text{dissipation of turbulent kinetic energy (TKE)}=\nu\overline{\|\grad  u'\|^2}, \\
    \mathcal{P}=\text{production of TKE}=\int_\Omega R( u, u):\grad\overline{ u} d x.
\end{gather*}
Hence $\frac{\mathcal{P}}{\epsilon}= 1$ implies $\int_\Omega R( u, u):\grad \overline{ u}\ d x= \int_\Omega \nu\overline{|\grad  u'|^2}\  d x$.
\end{remark}

\section{Model derivation}\label{sec:3}
In this section, we develop a model for turbulence not at statistical equilibrium 
unlike the Smagorinsky model \eqref{smag}.

Consider the Navier-Stokes Equations (NSE) which govern the flow of an incompressible fluid with velocity $ u( x,t)$, pressure $p( x,t)$, prescribed body force $f$ and kinematic viscosity $\nu$ in the regular and bounded flow domain $\Omega\subset  \mathbb{R}^d (d=2,\ 3)$:
\begin{equation}\label{nse1}
\begin{aligned}
 u_t+ u\cdot\grad u-\nu\Delta{ u}+\grad p=f( x) \ \text{in}\ \Omega,\ \text{and}\ \div u&=0 \ \text{in}\ \Omega.
\end{aligned}
\end{equation}
To derive the Corrected Smagorinsky model, following the work in \cite{jiang2016ev}, we begin with an ensemble of NSE solution $u(x,t;w_j)$ with perturbed initial data $ u( x,0;\omega_j)= u_0( x;\omega_j),\ j=1,\ 2,\dots,\ J,\  x\in\Omega.$ 

The goal of a turbulent model solution of \eqref{smag} and \eqref{CSM0} is to approximate $\overline{ u}( x, t)$.
By ensemble averaging the NSE gives a system that is not closed
since $\overline{ u u}\neq \bar{ u}\bar{ u}$. Hence the Reynolds stress tensor, $R( u, u)=\bar{ u}\bar{ u}-\overline{ u u}$ which is accountable for all effects of the fluctuations on the mean flow must be modeled \cite{rong2019extension}. We rewrite $\overline{ u u}$ as $\overline{ u u}=\bar{ u}\bar{ u}- R( u, u)$. Note that by using properties in \eqref{properties}, we get  $R=-\overline{ u' u'}$.
Hence we get,
\begin{equation}\label{ennse1}
\begin{aligned}
\overline{ u_t}+\bar{ u}\cdot\grad\bar{ u}-\nu\Delta{\bar{ u}}+\grad\bar{p}-\div R=f( x) \ \text{in}\ \Omega,\ \text{and}\ 
\div\bar{ u}=0 \ \text{in}\ \Omega.
\end{aligned}
\end{equation}
Take the dot product of first and second equation in (\eqref{ennse1}) with mean flow $\bar{ u}$ and $\bar{p}$ respectively and doing integration by parts, we get the energy estimate as follows \cite{rong2019extension,jiang2016ev}
\begin{align}\label{ennse3}
    \underbrace{\frac{1}{2}\frac{d}{dt}\|\bar{ u}\|^2}_{\text{rate of change of kinetic energy }}+\underbrace{\nu\|\grad\bar{ u}\|^2}_{\text{energy dissipation due to viscous forces}}\notag\\ 
    +\underbrace{\int_{\Omega}R( u, u):\grad\bar{ u}\ d x}_{\text{effect of fluctuation }} =\underbrace{(f,\bar{ u}).}_{\text{energy input through body force-flow interaction}}
\end{align}
In (\eqref{ennse3}), if the term $\int_{\Omega}R( u, u):\grad\bar{ u}\ d x>0$, the effect of $R( u, u)$ is dissipative while if $\int_{\Omega}R( u, u):\grad\bar{ u}\ d x<0$, 
fluctuations $ u'$ transfers energy back to mean $\Bar{ u}$ which causes increased energy in mean flow. 

Property 1 in \cref{b} is consistent with the assumption of Boussinesq \cite{bous} that turbulent fluctuations are dissipative on the mean in the time averaged case. In property 2 of \cref{b}, the term $\int_{\Omega}\nu\overline{|\grad u'|^2}\ d x$ is clearly dissipative while $\frac{d}{dt}\int_{\Omega}\overline{| u'|^2}\ d x=0$ for flows at statistical equilibrium. 
The idea of any EV model is based on three assumptions \cite{jiang2016ev}. 
Firstly, the statistical equilibrium assumption that dissipativity holds at each instant time 
\begin{equation}\label{statistical}
\int_{\Omega}R( u, u):\grad\bar{ u}\ d x\simeq \int_{\Omega}\nu\overline{|\grad u'|^2}\ d x.
\end{equation}
The second assumption is that $\grad u'$ aligns with $\grad\bar{ u}$. Third, calibration \cite{jiang2016ev} provides that the action of fluctuating velocities can be represented in terms of mean flow
\begin{equation*}
    action(\grad u')\simeq a(\bar{ u})\grad{\bar{ u}}.
\end{equation*}
Combining all these three assumptions
results in the eddy viscosity closure, 
\begin{equation*}
-\div R( u, u)\Leftarrow -\div(\nu_T(\bar{ u})\grad\bar{ u})+\text{terms incorporated in}\ \grad \bar{p}.
\end{equation*}
Here $\nu_T$ denotes the turbulent viscosity. 
Thus we have the eddy viscosity (EV) model: $\grad\cdot w=0 $ and
\begin{equation}\label{ev}
     w_t+ w\cdot\grad w-\nu\Dt w+\grad q-\div\left(\nu_T( w)\grad w\right)=f( x)
    .
\end{equation}
The solution $( w,q)$ of (\eqref{ev}) is an approximation of the ensemble average ($\bar{ u},\bar{p}$).
In 1963, Smagorinsky \cite{Smagorinsky} 
model $\nu_T$ by
\begin{equation}\label{nuT}
    \nu_T=(C_s\delta)^2|\grad  w|,
\end{equation}
where $C_s\approx 0.1$, 
Lilly \cite{lilly}. Let $\Dt x$ to be the mesh size and $\delta=\Delta x << 1$ 
is the model length scale \cite{smagorinsky1993}. Thus we get the classic Smagorinsky model \eqref{smag}.


By taking the dot product with $ w$, here we have the energy equality for Smagorinsky model:
\begin{equation*}
    \frac{1}{2}\frac{d}{dt}\| w\|^2+\nu\|\grad w\|^2+(C_s\delta)^2\|\grad w\|_{L^3}^3 =(f, w).
\end{equation*}
$(C_s\delta)^2\|\grad w\|^3\geq0 $ approximates the average energy dissipated by fluctuation.
Since it is positive, it prevents the energy from returning to the mean flow. 
We aim to remove this flaw in the corrected model. Notice that in \eqref{statistical}, the term $\frac{d}{dt}\int_{\Omega}\overline{| u'|^2}\ d x$ is omitted for flows at statistical equilibrium. This term is accountable for backscatter and other non-equilibrium effects. To model this term, we must express $u'$ in terms of $\bar{u}$. One idea in \cite{jiang2016ev} is that since the Smagorinsky model is dimensionally consistent, it must conform to the Kolmogorov-Prandtl relation  \cite{kolmogrov,prandtl}

\begin{equation}\label{k-p relation}
\nu_T=\mu l\sqrt{k'},
\end{equation}
where $\mu\approx 0.3 \ \text{to} \ 0.55$, $l=\text{turbulent length scale}$ and $k'$ is the turbulent kinetic energy:
$
    k'( x, t)=\frac{1}{2} \overline{| u'( x, t)|^2}.
    $
Thus, the Smagorinsky model contains a implicit model of $l$ and $k'$. Equating \eqref{nuT} with  \eqref{k-p relation} gives 
\begin{equation*}
    \mu l \sqrt{k'}=\nu_T=(C_s\delta)^2|\grad w|=\mu\delta\Big(\frac{C_s^2\delta}{\mu}|\grad w|\Big).
\end{equation*}
Here $\delta$ is the obvious choice for $l$. With
$\delta=l$, the Smagorinsky Model yields the model  $k'=C_s^4\delta^2\mu^{-2}|\grad w|^2$. Hence, the omitted term responsible for non-equilibrium effects is modeled as
\begin{align*}
\frac{d}{dt}\int_\Og k'\ d x
=\frac{d}{dt}C_s^4\delta^2\mu^{-2}(\grad w,\grad w)&=C_s^4\delta^2\mu^{-2}(-\Delta w_t, w).
\end{align*}
By including $C_s^4\delta^2\mu^{-2}\Delta  w_t$ in the model, its energy balance has a consistent representation of the term  $\frac{1}{2}\frac{d}{dt}\overline{\| u'\|^2}$.  
As a result, the Corrected Smagorinsky Model (CSM) is: $\div w=0$ and
\begin{equation}\label{CSM8}
 w_t-C_s^4\delta^2\mu^{-2}\Dt w_t+ w\cdot\grad w-\nu\Dt w+\grad q-\div\Big((C_s\delta)^2|\grad w|\grad w\Big)=f( x).
\end{equation}
Here we impose the no-slip boundary condition, $ w=0$ on $\partial\Og$.

\section{Basic Properties of the Model} \label{sec:4}
In this section, we develop some basic properties of the model which are useful in numerical analysis. In particular, we derive the basic energy estimate, we prove a stability bound and uniqueness of the solution. We also analyze the modeling error and numerical error of the model.
\subsection{Energy Estimate for the CSM}
We will 
identify the model's kinetic energy and energy dissipation in \cref{thm:energy}.
\begin{theorem}\label{thm:energy}
Let $ w$ be a strong solution of the Corrected Smagorinsky Model \eqref{CSM8}, then the following energy estimate holds
\begin{equation}\label{CSM9}
\begin{aligned}
    &\frac{1}{2}\frac{1}{|\Omega|}\frac{d}{dt}\Big(\| w\|^2+\frac{C_s^4\delta^2}{\mu^2}\|\grad w\|^2\Big)\\
    +\frac{1}{|\Omega|}\nu\|\grad w\|^2&+\frac{1}{|\Omega|}(C_s\delta)^2\|\grad w\|_{L^3}^3=\frac{1}{|\Omega|}\langle f, w\rangle.
\end{aligned}
\end{equation}
\end{theorem}
\begin{proof}
First, we take dot product in (\eqref{CSM8}) with $ w$ and do integration by parts which is shown below,
\begin{gather*}
    \int_{\Omega}\Big( w_t-C_s^4\delta^2\mu^{-2}\Dt w_t+ w\cdot\grad w-\nu\Dt w+\grad q-\div\Big((C_s\delta)^2|\grad w|\grad w\Big)\Big)\cdot  w d x\\
    =\int_{\Omega}f\cdot w\ d x.
\end{gather*}
Here, $\int_{\Omega} w_t\cdot w \ d x=\frac{d}{dt}\Big(\frac{1}{2}\int_{\Omega}| w|^2\ d x\Big).$
By Lemma \eqref{trilinear1},
$
    \int_\Omega  w\cdot\grad w\cdot w\ d x=0.
$
Next, $-\nu\int_{\Omega}\Delta w\cdot w \ d x
=\int_{\Omega}\nu|\grad w|^2\ d x.$
The next term,
$
  \int_{\Omega}\grad q\cdot  w \ d  x
  =\int_{\partial\Omega}{p w\cdot\hat{n}}\ \ ds-\int_{\Omega}{p \div  w }\ d  x=0.
$
The final term,
\begin{equation*}
    \int_\Omega -\grad\cdot\left((C_s\delta)^2|\grad w|\grad w\right) w\ d x=\int_\Omega (C_s\delta)^2|\grad w|\grad w\cdot\grad w\ dx=\int_\Omega(C_s\delta)^2|\grad w|^3\ d x.
\end{equation*}
Hence combining all these terms we get the following energy estimate per unit volume,
\begin{equation*}
\begin{aligned}
    &\frac{1}{2}\frac{1}{|\Omega|}\frac{d}{dt}\Big(\| w\|^2+\frac{C_s^4\delta^2}{\mu^2}\|\grad w\|^2\Big)\\
    +\frac{1}{|\Omega|}\nu\|\grad w\|^2&+\frac{1}{|\Omega|}(C_s\delta)^2\|\grad w\|_{L^3}^3=\frac{1}{|\Omega|}\langle f, w\rangle.
\end{aligned}
\end{equation*}
\end{proof}
\begin{remark}
In \eqref{CSM9}, we can identify the following quantities:
\begin{enumerate}
    \item Model kinetic energy of mean flow per unit volume
    \begin{equation*}
    MKE:=\frac{1}{2}\frac{1}{|\Omega|}\Big(\| w\|^2+\frac{C_s^4}{\mu^2}\delta^2\|\grad w\|^2\Big).
    \end{equation*}
    And the second term in MKE coming from the Corrected Smagorinsky Model is the turbulent kinetic energy per unit volume.
    \item Rate of energy dissipation of mean flow per unit volume
    \begin{equation*}
    \varepsilon_{CSM}(t):=\frac{1}{|\Omega|}\Big(\nu\|\grad w\|^2+(C_s\delta)^2\|\grad w\|_{L^3}^3\Big).
    \end{equation*}
    This controls the time rate of change of kinetic energy. It's always positive and it reduces the accumulation of kinetic energy.
\item Rate of energy input to mean flow per unit volume is
$\frac{1}{|\Omega|}\langle f, w\rangle.$
\end{enumerate}
\end{remark}
\subsection{Stability}
Next, we give the stability bound of the Corrected Smagorinsky Model \eqref{CSM8} in \cref{thm:stability}. We prove the model kinetic energy is bounded uniformly in time and the time-averaged model energy dissipation rate is bounded as well in the same Theorem.
\begin{theorem}\label{thm:stability}
(Stability of $ w$) \eqref{CSM8} is unconditionally stable. The solution $ w$ of \eqref{CSM8}  satisfies the following inequality
\begin{equation*}
     \| w(T)\|^2+\frac{C_s^4}{\mu^2}\delta^2\|\grad w(T)\|^2\leq e^{-\alpha T}\Big\{\| w(0)\|^2+\frac{C_s^4}{\mu^2}\delta^2\|\grad w(0)\|^2+\frac{C}{\alpha}(e^{\alpha T}-1)\Big\},
\end{equation*}
where 
$
\alpha=\min\{\frac{\nu}{2C_{PF}^2},\ \frac{\mu^2\nu}{C_s^4\delta^2}\},
$
and 
if $f\in L^2(\Omega)$, we get $$\max_{0\leq t<\infty}\Big(\| w\|^2+\frac{C_s^4\delta^2}{\mu^2}\|\grad w\|^2\Big)\leq C'< \infty.$$
and
\begin{equation*}
    \mathcal{O}(\frac{1}{T})+\frac{1}{|\Omega|}\frac{1}{T}\int_0^T\Big(\frac{\nu}{2}\|\grad w\|^2+(C_s\delta)^2\|\grad w\|_{L^3}^3\Big)\ dt\leq \frac{1}{|\Omega|}\frac{1}{T}\int_0^T \frac{C_{PF}^2}{2\nu}\|f\|^2\ dt.
\end{equation*}
\end{theorem}
\begin{proof}
Take $L^2$ inner product of \eqref{CSM8} with $ w$, we get the following energy equality,
\begin{equation}\label{energy equality}
    \frac{1}{2}\frac{d}{dt}\Big(\| w\|^2+\frac{c_s^4}{\mu^2}\delta^2\|\grad w\|^2\Big)+\nu\|\grad w\|^2+(C_s\delta)^2\|\grad w\|_{L^3}^3=(f, w).
\end{equation}
Consider the RHS of \eqref{energy equality}, 
$
    (f, w)\leq \frac{\epsilon}{2}\| w\|^2+\frac{1}{2\epsilon}\|f\|^2.
$
Thus \eqref{energy equality} implies
\begin{equation*}
   \frac{d}{dt}\Big(\| w\|^2+\frac{C_s^4}{\mu^2}\delta^2\|\grad w\|^2\Big)+\nu\|\grad w\|^2+\nu\|\grad w\|^2+2(C_s\delta)^2\|\grad w\|_{L^3}^3\leq
   \epsilon\| w\|^2+\frac{1}{\epsilon}\|f\|^2.
\end{equation*}
Using the inequality $\| w\|\leq C_{PF}\|\grad w\|$, we have
\begin{equation*}
    \frac{d}{dt}\Big(\| w\|^2+\frac{C_s^4}{\mu^2}\delta^2\|\grad w\|^2\Big)+\frac{\nu}{C_{PF}^2}\| w\|^2+\nu\|\grad w\|^2+2(C_s\delta)^2\|\grad w\|_{L^3}^3\leq
    \epsilon\| w\|^2+\frac{1}{\epsilon}\|f\|^2.
\end{equation*}
Pick $\epsilon=\frac{\nu}{2C_{PF}^2}$ and drop the term $2(C_s\delta)^2\|\grad w\|_{L^3}^3$. We obtain
\begin{align*}
    \frac{d}{dt}\Big(\| w\|^2+\frac{C_s^4}{\mu^2}\delta^2\|\grad w\|^2\Big)+\frac{\nu}{2C_{PF}^2}\| w\|^2+\frac{\mu^2}{C_s^4\delta^2}\nu\Big(\frac{C_s^4}{\mu^2}\delta^2\|\grad w\|^2\Big)&\leq \frac{2}{\nu}C_{PF}^2\|f\|^2 .
\end{align*}
Let $\alpha=\min\{\frac{\nu}{2C_{PF}^2},\ \frac{\mu^2}{C_s^4\delta^2}\nu\}$, resulting in
\begin{align*}
    \frac{d}{dt}\Big(\| w\|^2+\frac{C_s^4}{\mu^2}\delta^2\|\grad w\|^2\Big)+\alpha \Big(\| w\|^2+\frac{C_s^4}{\mu^2}\delta^2\|\grad w\|^2\Big)\leq \frac{2}{\nu}C_{PF}^2\|f\|^2.
\end{align*}
Multiply by the integrating factor $e^{\alpha t}$ and integrate from $t=0$ to $t=T$, leading to
\begin{equation*}
    \| w(T)\|^2+\frac{C_s^4}{\mu^2}\delta^2\|\grad w(T)\|^2\leq e^{-\alpha T}\Big\{\| w(0)\|^2+\frac{C_s^4}{\mu^2}\delta^2\|\grad w(0)\|^2+\frac{C}{\alpha}(e^{\alpha T}-1)\Big\},
\end{equation*}
where $C=\frac{2}{\nu}C_{PF}^2\|f\|^2$.\\
This implies that kinetic energy is uniformly bounded, i.e. if $f\in L^2(\Omega)$, we get 
\begin{equation*}
\max_{0\leq t<\infty}\Big(\| w\|^2+\frac{C_s^4\delta^2}{\mu^2}\|\grad w\|^2\Big)\leq C'< \infty.
\end{equation*}
Integrate \eqref{energy equality} from $t=0$ to $t=T$ and divide by $|\Omega|$ and $T$, we have
\begin{equation}\label{ee}
\begin{aligned}
    \frac{1}{|\Omega|}\frac{1}{2T}\Big\{\Big(\| w(T)\|^2+\frac{C_s^4}{\mu^2}\delta^2\|\grad w(T)\|^2\Big)-\Big(\| w(0)\|^2+\frac{C_s^4}{\mu^2}\delta^2\|\grad w(0)\|^2\Big)\Big\} \\
    +\frac{1}{|\Omega|}\frac{1}{T}\int_0^T\Big(\nu\|\grad w\|^2+(C_s\delta)^2\|\grad w\|_{L^3}^3\Big)\ dt=\frac{1}{|\Omega|}\frac{1}{T}\int_0^T(f, w)\  dt.
\end{aligned}
\end{equation}
Consider the term on the right. Using the Poincar\'{e}-Friedrichs' inequality \eqref{pf ineq}, Cauchy Schwarz and Young's inequality 
gives
\begin{align*}
    \frac{1}{|\Omega|}\frac{1}{T}\int_0^T(f, w)\  dt&\leq \frac{1}{|\Omega|}\frac{1}{T}\int_0^T \frac{1}{\sqrt{\nu}}\|f\|C_{PF}\sqrt{\nu}\|\grad  w\|\ dt ,\\
    &\leq \frac{1}{|\Omega|}\frac{1}{T}\int_0^T \frac{\nu}{2}\|\grad  w\|^2\  dt+\frac{1}{|\Omega|}\frac{1}{T}\int_0^T \frac{C_{PF}^2}{2\nu}\|f\|^2\ dt.
\end{align*}
The first term in \eqref{ee} is bounded by the previous result. Thus,
\begin{equation*}
    \mathcal{O}(\frac{1}{T})+\frac{1}{|\Omega|}\frac{1}{T}\int_0^T\Big(\frac{\nu}{2}\|\grad w\|^2+(C_s\delta)^2\|\grad w\|_{L^3}^3\Big)\ dt\leq \frac{1}{|\Omega|}\frac{1}{T}\int_0^T \frac{C_{PF}^2}{2\nu}\|f\|^2\ dt.
\end{equation*}
The time-averaged dissipation is bounded.
\end{proof}
\subsection{Uniqueness}
In this subsection, we prove the uniqueness of the strong solution to \eqref{CSM8} in \cref{thm:unique}.
\begin{theorem}\label{thm:unique}
Assume $\grad w\in L^3(0,T;L^3(\Omega))$, the solution $ w$ of \eqref{CSM8} is unique.
\end{theorem}
\begin{proof}
Suppose $( w_1,q_1)$ and $( w_2,q_2)$ are two different solution of \eqref{CSM8} and let $ \phi,\ q$ denote the difference between two solutions: $ \phi= w_1- w_2$,  $q=q_1-q_2$,  $ \phi$ satisfies $\grad\cdot \phi=0$ and
\begin{align*}
    \frac{\partial}{\partial t}\left( \phi-\frac{C_s^4}{\mu^2}\delta^2\Delta \phi\right)+ w_1\cdot\grad w_1- w_2\cdot\grad w_2-\nu\Delta \phi+\grad q\\-(C_s\delta)^2\grad\cdot(|\grad w_1|\grad w_1-|\grad w_2|\grad w_2)=0 .
\end{align*}
Take the $L^2$ inner product with $ \phi$ and  let $\Tilde{ w}$ represent either $ w_1$ or $ w_2$, we obtain
\begin{gather*}
    \frac{1}{2}\frac{d}{dt}\left(\| \phi\|^2+\frac{C_s^4}{\mu^2}\delta^2\|\grad \phi\|^2\right)+\nu\|\grad \phi\|^2\\
    +(C_s\delta)^2\int_\Omega[|\grad w_1|\grad w_1-|\grad w_2|\grad w_2]\cdot\grad( w_1- w_2)\ d x=-\int_\Omega \phi\cdot\grad\Tilde{ w}\cdot \phi\ d x.
\end{gather*}
Using the Strong Monotonicity \eqref{strong-monotonicity}, we get
\begin{equation*}
    \frac{1}{2}\frac{d}{dt}\left(\| \phi\|^2+\frac{C_s^4}{\mu^2}\delta^2\|\grad \phi\|^2\right)+\nu\|\grad \phi\|^2+C_1(C_s\delta)^2\|\grad \phi\|_{L_3}^3\leq -\int_\Omega \phi\cdot\grad w\cdot \phi\ d x.
\end{equation*}
Consider the RHS, using \eqref{soblev2} in 3D space,
\begin{align*}
    \left|-\int_\Omega \phi\cdot\grad w\cdot \phi\ d x\right| &\leq\|\grad w\|_{L^3}\| \phi\|_{L^3}^2, \\
    &\leq C\|\grad w\|_{L^3}(\| \phi\|^{1/2}\|\grad \phi\|^{1/2})^2, \\
    &\leq \frac{\epsilon}{2}\|\grad \phi\|^2+C(\epsilon)\|\grad w\|_{L^3}^2\| \phi\|^2.
\end{align*}
Pick $\epsilon=2\frac{C_s^4}{\mu^2}\delta^2$, leading to
\begin{align*}
    \frac{1}{2}\frac{d}{dt}(\|\bm{ \phi}\|^2+\frac{C_s^4}{\mu^2}\delta^2\|\grad \phi\|^2)&+\nu\|\grad \phi\|^2+C_1(C_s\delta)^2\|\grad \phi\|_{L^3}^3\\
    &\leq \frac{C_s^4}{\mu^2}\delta^2\|\grad \phi\|^2+C(\epsilon)\|\grad w\|_{L^3}^2\| \phi\|^2 ,\\
    &\leq \max\{1,C(\epsilon)\|\grad w\|_{L^3}^2\}(\frac{C_s^4}{\mu^2}\delta^2\|\grad \phi\|^2+\| \phi\|^2).
\end{align*}
Here $a(t):=\max\{1,\ C(\epsilon)\|\grad w_1\|_{L^3}^2\}\in L^1(0,T)$,\ because 
\begin{align*}
\int_0^T 1\cdot \|\grad w_1\|_{L^3}^2 \ dt&\leq \left(\int_0^T 1^3\ dt\right)^{1/3}\cdot \left(\int_0^T\|\grad w_1\|_{L^3}^{2\times \frac{3}{2}}\ dt\right)^{2/3}\\
&=\left(\int_0^T 1^3\ dt\right)^{1/3}\cdot \left(\int_0^T\|\grad w_1\|_{L^3}^{3}\ dt\right)^{2/3}< \infty.
\end{align*}
Then we can form its antiderivative
\begin{align*}
    A(T):=\int_0^T a(t)\ dt<\infty,\ \text{for}\ \grad w\in L^3(0,T;L^3(\Omega)).
\end{align*}
Multiplying through by the integrating factor $e^{-A(t)}$ gives
\begin{align*}
    \frac{d}{dt}\left[\frac{1}{2}e^{-A(t)}\left(\| \phi\|^2+\frac{C_s^4\delta^2}{\mu^2}\|\grad \phi\|^2\right)\right]+e^{-A(t)}[\nu\|\grad \phi\|^2+C_1(C_s\delta)^2\|\grad \phi\|_{L^3}^3] \leq 0.
\end{align*}
Then, integrating over $[0,T]$ and multiplying through by $e^{A(t)}$ gives
\begin{align*}
    \frac{1}{2}\Big(\| \phi(T)\|^2+\frac{C_s^4}{\mu^2}\delta^2\|\grad \phi(T)\|^2\Big)+\int_0^T \Big(\nu\|\grad \phi\|^2+C_1(C_s\delta)^2\|\grad \phi\|_{L^3}^3\Big)\ dt\\\leq e^{A(t)}\frac{1}{2}\Big(\| \phi(0)\|^2+\frac{C_s^4}{\mu^2}\delta^2\|\grad \phi(0)\|^2\Big).
\end{align*}
\end{proof}
\subsection{Modelling error}
In this subsection, we analyze the error between the solution to the NSE \eqref{nse1} and the Corrected Smagorinksy Model \eqref{CSM8} in \cref{thm:error}.
\begin{theorem}\label{thm:error}
Assume $\grad  u_t \in L^2(\Omega)$ and $
\grad w\in L^2(0,T;L^3)$, let $ \phi= u_{NSE}- w_{Smag}$ be the modelling error of the Corrected Smagorinsky, then $ \phi$ satisfies the following:
\begin{gather*}
    \frac{1}{2}\Big(\| \phi(T)\|^2+\frac{C_s^4}{\mu^2}\delta^2\|\grad \phi(T)\|^2\Big)+\int_0^T \frac{\nu}{2}\|\grad \phi\|^2+\frac{C_1}{2}(C_s\delta)^2\|\grad \phi\|_{L^3}^3\ dt \\\leq
    C^*\left\{\frac{1}{2}\Big(\| \phi(0)\|^2+\frac{C_s^4}{\mu^2}\delta^2\|\grad \phi(0)\|^2\Big)+\int_0^T (C_s\delta)^2\|\grad u\|_{L^3}^3+\frac{C_s^4}{\mu^2}\delta^2\|\grad u_t\|^2\ dt\right\}.
\end{gather*}
Here $C^*$ depends on $\nu,\ T,\ \int_0^T \|\grad w\|_{L^3}^2\ dt$. 
\end{theorem}
\begin{proof}
$ u_{NSE}$ satisfies $\grad\cdot u =0$ and the following equation
\begin{equation}\label{unse}
\begin{aligned}
     u_t+ u\cdot\grad u-\nu\Delta u+\grad p-(C_s\delta)^2\grad\cdot(|\grad u|\grad u)-\frac{C_s^4}{\mu^2}\delta^2\Delta u_t\\
    =f-(C_s\delta)^2\grad\cdot(|\grad u|\grad u)-\frac{C_s^4}{\mu^2}\delta^2\Delta u_t.
\end{aligned}
\end{equation}
Subtract \eqref{CSM8} from \eqref{unse}. We obtain, $\grad \cdot  \phi=0$ and
\begin{gather*}
    { \phi}_t-\frac{C_s^4}{\mu^2}\delta^2\Delta{ \phi}_t+ u\cdot\grad u- w\cdot\grad w-\nu\Delta \phi+ \grad(p-q)\\
   -(C_s\delta)^2\grad\cdot(|\grad u|\grad u-|\grad w|\grad w) =-(C_s\delta)^2\grad\cdot(|\grad u|\grad u)-\frac{C_s^4}{\mu^2}\delta^2\Delta u_t.
\end{gather*}
Here, $ u\cdot\grad u- w\cdot\grad w= u\cdot\grad u- u\cdot \grad w+ u\cdot \grad w- w\cdot\grad w= u\cdot\grad \phi+ \phi\cdot\grad w$.\\
Take $L^2$ inner product with $ \phi$ gives
\begin{gather*}
    \frac{1}{2}\frac{d}{dt}\Big(\| \phi\|^2+\frac{C_s^4}{\mu^2}\delta^2\|\grad \phi\|^2\Big)+\int_\Omega  \phi\cdot\grad w\cdot \phi\  d x+\nu\|\grad \phi\|^2\ d x\\
    +(C_s\delta)^2\int_\Omega (|\grad u|\grad u-|\grad w|\grad w)\cdot\grad( u- w)\ d x\\
    =(C_s\delta)^2\int_\Omega |\grad u|\grad u:\grad \phi\  d x+\frac{C_s^4}{\mu^2}\delta^2\int_\Omega \grad u_t:\grad \phi\  d x.
\end{gather*}
Using Strong Monotonicity \eqref{monotonicity and LLC}, we have
\begin{equation}\label{difference-equality}
\begin{aligned}
    &\frac{1}{2}\frac{d}{dt}\Big(\| \phi\|^2+\frac{C_s^4}{\mu^2}\delta^2\|\grad \phi\|^2\Big)+\nu\|\grad \phi\|^2+C_1(C_s\delta)^2\|\grad \phi\|_{L^3}^3\\
    \leq-\int_\Omega & \phi\cdot\grad w\cdot \phi\ d  x+(C_s\delta)^2\int_\Omega |\grad u|\grad u:\grad \phi\  d x+\frac{C_s^4}{\mu^2}\delta^2\int_\Omega \grad u_t:\grad \phi\  d x.
\end{aligned}
\end{equation}
Consider the first term in the RHS, similar to the previous steps
\begin{align*}
    \Big|-\int_\Omega  \phi\cdot\grad w\cdot \phi\  d x\Big|\leq \frac{\epsilon_1}{2}\|\grad \phi\|^2+C(\epsilon_1)\|\grad w\|_{L^3}^2\| \phi\|^2.
\end{align*}
The second term in the RHS is
\begin{align*}
    \Big|(C_s\delta)^2\int_\Omega |\grad u|\grad u:\grad \phi\  d x\Big|&\leq (C_s\delta)^2\|\grad \phi\|_{L^3}\|\grad u\|_{L^3}^2,\\
    &\leq \frac{\epsilon_2}{3}(C_s\delta)^2\|\grad \phi\|_{L^3}^3+C(\epsilon_2)(C_s\delta)^2(\|\grad u\|_{L^3}^2)^{3/2},\\
    &= \frac{\epsilon_2}{3}(C_s\delta)^2\|\grad \phi\|_{L^3}^3+C(\epsilon_2)(C_s\delta)^2\|\grad u\|_{L^3}^3.
\end{align*}
The third term in the RHS satisfies
\begin{align*}
    \Big|\frac{C_s^4}{\mu^2}\delta^2\int_\Omega \grad u_t:\grad \phi\  d x\Big|&\leq\frac{C_s^4}{\mu^2}\delta^2\|\grad u_t\|\|\grad \phi\|,\\
    &\leq \frac{\epsilon_3}{2}\frac{C_s^4}{\mu^2}\delta^2\|\grad \phi\|^2+C(\epsilon_3)\frac{C_s^4}{\mu^2}\delta^2\|\grad u_t\|^2.
\end{align*}
Pick $\epsilon_1=\nu,\epsilon_2=\frac{3C_1}{2}$, collect all terms, \eqref{difference-equality} becomes
\begin{gather*}
    \frac{1}{2}\frac{d}{dt}\Big(\| \phi\|^2+\frac{C_s^4}{\mu^2}\delta^2\|\grad \phi\|^2\Big)+\frac{\nu}{2}\|\grad \phi\|^2+\frac{C_1}{2}(C_s\delta)^2\|\grad \phi\|_{L^3}^3\\
    \leq\max\left\{C(\epsilon_1)\|\grad w\|_{L^3}^2,\frac{\epsilon_3}{2}\right\}(\| \phi\|^2+\frac{C_s^4}{\mu^2}\delta^2\|\grad \phi\|^2)\\
    +C(\epsilon_2)(C_s\delta)^2\|\grad u\|_{L^3}^3+C(\epsilon_3)\frac{C_s^4}{\mu^2}\delta^2\|\grad u_t\|^2.
\end{gather*}
Denote $a(t):=\max\left\{C(\epsilon_1)\|\grad w\|_{L^3}^2,\ \frac{\epsilon_3}{2}\right\}$ and its antiderivative is given by
\begin{align*}
    A(T):=\int_0^T a(t)\ dt<\infty\ \text{for} 
    \ \grad w\in L^2(0,T;L^3).
\end{align*}
Multiplying through by the integrating factor $e^{-A(t)}$ gives
\begin{gather*}
    \frac{d}{dt}\left[\frac{1}{2}e^{-A(t)}\left(\| \phi\|^2+\frac{C_s^4\delta^2}{\mu^2}\|\grad \phi\|^2\right)\right]+e^{-A(t)}\left[\frac{\nu}{2}\|\grad \phi\|^2+\frac{C_1}{2}(C_s\delta)^2\|\grad \phi\|_{L^3}^3\right] \\
    \leq e^{-A(t)}\left\{C(\epsilon_2)(C_s\delta)^2\|\grad u\|_{L^3}^3+C(\epsilon_3)\frac{C_s^4\delta^2}{\mu^2}\|\grad u_t\|^2\right\}.
\end{gather*}
Then, integrating over $[0,T]$ and multiplying through by $e^{A(t)}$ gives
\begin{gather*}
    \frac{1}{2}\Big(\| \phi(T)\|^2+\frac{C_s^4}{\mu^2}\delta^2\|\grad \phi\|^2\Big)+\int_0^T \frac{\nu}{2}\|\grad \phi\|^2+\frac{C_1}{2}(C_s\delta)^2\|\grad \phi\|_{L^3}^3\ dt \\\leq
    C(\nu,T,\|\grad w\|_{L^3})\Big\{\frac{1}{2}\Big(\| \phi(0)\|^2+\frac{C_s^4}{\mu^2}\delta^2\|\grad \phi(0)\|^2\Big)\\
    +\int_0^T (C_s\delta)^2\|\grad u\|_{L^3}^3+\frac{C_s^4}{\mu^2}\delta^2\|\grad u_t\|^2\ dt\Big\},
\end{gather*}
and $C(\nu,T)$ depends on $\nu,\ T,\ \int_0^T \|\grad w\|_{L^3}^2\ dt$. 
\end{proof}
\section{Numerical error}\label{sec:5}
Consider the semi-discrete approximation of the CSM \eqref{CSM8}  with grad-div stabilization. 
Suppose $ w^h( x,0)$ is approximation of $ w( x,0)$. 
The approximate velocity and pressure are maps
\begin{equation*}
     w^h: [0,T]\to X^h\quad, p^h:(0,T]\to Q^h 
\end{equation*}
\textit{satisfying 
}
\begin{equation}\label{semi-CSM}
\begin{aligned}
    ( w^h_t, v^h) +\frac{C_s^4}{\mu^2}\delta^2 (\grad w^h_t,\grad  v^h) +b^*( w^h, w^h, v^h)+\nu(\grad w^h,\grad  v^h) +\gamma(\grad\cdot w^h,\grad\cdot v^h)\\
    -( p^h,\grad\cdot  v^h)+\Big((C_s\delta)^2|\grad w^h|\grad w^h,\grad  v^h\Big)=(f, v^h)\ \text{for all}\  v^h\in X^h,\\
    (q^h,\grad\cdot w^h)=0\ \text{for all}\ q^h\in Q^h.
\end{aligned}
\end{equation}

In this section, we analyze the error between the strong solution to the CSM \eqref{CSM8} and the semi-discrete solution to \eqref{semi-CSM} in \cref{thm:numerial-error}.
\begin{theorem}\label{thm:numerial-error}
(Numerical error of semi-discrete case) Let $ w$ be the strong solution of the CSM \eqref{CSM8} (in particular $\| w\|\in L^{\infty}(0,T),\ \|\grad w\|_{L^2}\in L^{2}(0,T),\ 
 w\in L^2(0,T;W^{1,3}(\Omega))\cap L^2(0,T;L^6(\Omega))$) and $ w^h$ be a solution to the semi-discrete problem \eqref{semi-CSM}. Let
\begin{equation*}
a(t):=C(\nu)\|\grad w\|_{L^3}^2+\frac{1}{4}\| w\|_{L^6}^2.
\end{equation*}
Then, for $T>0$ the error $ w- w^h$ satisfies 

\begin{gather*}
    \|( w- w^h)(T)\|^2+\frac{C_s^4\delta^2}{\mu^2}\|\grad( w- w^h)(T)\|^2\\+\int_0^T\Big\{ \frac{\nu}{4}\|\grad( w- w^h)\|^2+\gamma\|\grad\cdot(w-w^h)\|^2 +\frac{2}{3}C_1 (C_s\delta)^2\|\grad( w- w^h)\|_{L^3}^3\Big\}\ dt\\
    \leq \exp(\int_0^T a(t) dt)\Big\{ \|( w- w^h)(0)\|^2\\+\frac{C_s^4\delta^2}{\mu^2}\|\grad( w- w^h)(0)\|^2+\inf_{ v^h \in V^h}\|( w- v^h)(T)\|^2 \\
    + \int_0^T\Big[   C(\nu)\inf_{ v^h\in V^h}\left(\| w_t- v^h_t\|_{-1}^2 + (\frac{C_s^4\delta^2}{\mu^2})^2\|\grad( w_t- v^h_t)\|^2 + \|\grad ( w- v^h)\|^2\right)\\
     + C\inf_{ v^h\in V^h}\left((C_s\delta)^2\|\grad( w- v^h)\|_{L^3}^{3/2} + \delta^{-1}\|  w- v^h\|_{L^6}^{3/2}+\gamma\|\grad\cdot(w-v^h)\|^2\right)\\
      +C(\gamma^{-1})\inf_{q^h\in Q^h}\|p-q^h\|^2\Big]  dt  \Big\}.
\end{gather*}
\end{theorem}
\begin{proof}
Consider the variational problem of the CSM \eqref{CSM8}: Find $ w:[0,T]\to X=L^{\infty}(0,T;L^2(\Omega))\cap L^3(0,T;W^{1,3}(\Omega))$ satisfying \eqref{cont-CSM}.
Let $ v^h\in V^h=\{ v^h\in X^h:(\grad\cdot  v^h,q^h)=0\ \forall\ q^h\in Q^h\}$. 
Since $ v\in X$\  \& \ 
$ v^h\in V^h\subset X^h\subset X$,
we restrict $ v= v^h$ in continuous variational problem. Then subtract semi-discrete problem \eqref{semi-CSM} from continuous problem \eqref{cont-CSM}. Let $\bm{e}=error= w- w^h$. This gives, 
\begin{equation}\label{CSM21}
\begin{aligned}
  (\bm{e}_t, v^h)+(C_s^4&\delta^2\mu^{-2}\grad \bm{e}_t,\grad v^h)
  +  b^{*}( w,  w, v^h)-b^{*}( w^h,  w^h, v^h)\\+\nu(\grad \bm{e},\grad v^h) +\gamma(\grad\cdot e,\grad\cdot v^h)
  &+(C_s\delta)^2\int_{\Omega}(|\grad w|\grad w-|\grad w^h|\grad w^h):\grad v^h \ d x\\
  &-(p-p^h,\div  v^h)=0.
\end{aligned}
\end{equation} 
We can write,
\begin{align*}
&b^{*}( w,  w, v^h)-b^{*}( w^h,  w^h, v^h)\\=&b^{*}( w,  w, v^h)-b^{*}( w^h,  w, v^h)+b^{*}( w^h,  w, v^h)-b^{*}( w^h,  w^h, v^h),\\=&b^{*}(\bm{e},  w, v^h)+b^{*}( w^h, \bm{e}, v^h).
\end{align*}
Also,
\begin{gather*}
    \int_{\Omega}(|\grad w|\grad w-|\grad w^h|\grad w^h)\cdot\grad v^h \ d x\\
    =\int_{\Omega}(|\grad w|\grad w-|\grad \widetilde{ w}|\grad \widetilde{ w}+|\grad \widetilde{ w}|\grad \widetilde{ w}-|\grad w^h|\grad w^h)\cdot\grad v^h \ d x.
\end{gather*}
Pick 
$\widetilde{ w}\in V^h$. 
Let $\bm{\eta}= w- \widetilde{ w},\ 
 \phi^h= w^h-\widetilde{ w},\ 
 \phi^h\in V^h$.
This implies $\bm{e}=( w-\widetilde{ w})-( w^h-\widetilde{ w})=\bm{\eta}- \phi^h$.
Then $\eqref{CSM21}$ becomes
\begin{gather*}
  ( \phi_t^h, v^h)+(C_s^4\delta^2\mu^{-2}\grad  \phi_t^h,\grad v^h)+b^{*}(\bm{e},  w, v^h)+b^{*}( w^h, \bm{e}, v^h)\\+\nu(\grad  \phi^h,\grad v^h) +\gamma(\grad\cdot\phi^h,\grad\cdot v^h)\notag \\+(C_s\delta)^2\int_{\Omega}(|\grad w^h|\grad w^h-|\grad\widetilde{ w}|\grad\widetilde{ w}):(\grad v^h) \ d x-(p-p^h,\div  v^h)\notag\\=(\bm{\eta}_t, v^h)+(C_s^4\delta^2\mu^{-2}\grad \bm{\eta}_t,\grad v^h)+\nu(\grad \bm{\eta},\grad v^h) +\gamma(\grad\cdot\eta,\grad\cdot v^h)\\
  +(C_s\delta)^2\int_{\Omega}(|\grad w|\grad w-|\grad\widetilde{ w}|\grad\widetilde{ w}):(\grad v^h) \ d x.
\end{gather*}
Take $ v^h= \phi^h$ and $\lambda^h\in Q^h$ .
Here $b^{*}( w^h, \bm{e}, \phi^h)=b^{*}( w^h, \bm{\eta}- \phi^h, \phi^h)=b^{*}( w^h, \bm{\eta}, \phi^h)$\  since\ $b^{*}( w^h,  \phi^h, \phi^h)=0$.
Using strong monotonocity \eqref{monotonicity and LLC}, we get
\begin{equation*}
(C_s\delta)^2\int_{\Omega}(|\grad w^h|\grad w^h-|\grad\widetilde{ w}|\grad\widetilde{ w}):(\grad \phi^h) \ d x\geq C_1(C_s\delta)^2\|\grad \phi^h\|_{L^3}^3.
\end{equation*}
Using local Lipschitz continuity \eqref{monotonicity and LLC}, we get 
\begin{align*}
&(C_s\delta)^2\int_{\Omega}(|\grad w|\grad w-|\grad\widetilde{ w}|\grad\widetilde{ w}):(\grad \phi^h) \ d x\leq (C_s\delta)^2C_2 r \|\grad\bm{\eta}\|_{L^3}\|\grad \phi^h\|_{L^3},\ \\ &where \ \ r=\max\{\|\grad w\|_{L^3},\ \|\grad\widetilde{ w}\|_{L^3}\}.
\end{align*}
Hence,
\begin{gather*}
    \frac{1}{2}\frac{d}{dt}\left\{\| \phi^h\|^2+\frac{C_s^4\delta^2}{\mu^2}\|\grad \phi^h\|^2\right\}+\nu\|\grad \phi^h\|^2 +\gamma\|\grad\cdot\phi^h\|^2\\+b^{*}(\bm{\eta}- \phi^h,  w, \phi^h)+b^{*}( w^h, \bm{\eta}, \phi^h)+C_1(C_s\delta)^2\|\grad \phi^h\|_{L^3}^3\\\leq (\bm{\eta}_t, \phi^h)+(C_s^4\delta^2\mu^{-2}\grad \bm{\eta}_t,\grad \phi^h)+\nu(\grad \bm{\eta},\grad \phi^h) +\gamma(\grad\cdot\eta,\grad\cdot\phi^h)\\+(C_s\delta)^2C_2 r \|\grad\bm{\eta}\|_{L^3}\|\grad \phi^h\|_{L^3}+(p-\lambda^h,\div \phi^h).
\end{gather*}
We can rewrite it as
\begin{gather*}
    \frac{1}{2}\frac{d}{dt}\left\{ \| \phi^h\|^2+\frac{C_s^4}{\mu^2}\delta^2\|\grad \phi^h\|^2\right\}+\nu\|\grad \phi^h\|^2 +\gamma\|\grad\cdot\phi^h\|^2+C_1(C_s\delta)^2\|\grad \phi^h\|_{L^3}^3\\
    \leq (\bm{\eta}_t, \phi^h)+\frac{C_s^4}{\mu^2}\delta^2(\grad\bm{\eta}_t,\grad \phi^h)+\nu(\grad\bm{\eta},\grad \phi^h)+(p-\lambda^h,\grad\cdot  \phi^h) +\gamma(\grad\cdot\eta,\grad\cdot\phi^h)\\
    +(C_s\delta)^2C_2r\|\grad\bm{\eta}\|_{L^3}\|\grad \phi^h\|_{L^3}-b^*(\bm{\eta}, w, \phi^h)+b^*( \phi^h, w, \phi^h)-b^*( w^h,\bm{\eta}, \phi^h).
\end{gather*}
Next we find the bounds for the terms in the RHS. For the first five terms on the right, use the Cauchy Schwarz 
and Young's inequality, 
\begin{equation*}
    |(\bm{\eta}_t, \phi^h)|\leq \|\bm{\eta}_t\|_{-1}\|\grad \phi^h\|
    \leq \frac{\nu}{2}\|\grad \phi^h\|^2+C(\nu)\|\bm{\eta}_t\|_{-1}^2.
\end{equation*}
\begin{align*}
    \frac{C_s^4}{\mu^2}\delta^2|(\grad\bm{\eta}_t,\grad \phi^h)|&\leq \|\grad \phi^h\|\frac{C_s^4\delta^2}{\mu^2}\|\grad\bm{\eta}_t\|, \\
    &\leq \frac{\nu}{4}\|\grad  \phi^h\|^2+C(\nu)\left(\frac{C_s^4\delta^2}{\mu^2}\right)^2\|\grad\bm{\eta}_t\|^2.
\end{align*}
\begin{equation*}
    \nu|(\grad \bm{\eta},\grad  \phi^h)|\leq \nu\|\grad\bm{\eta}\|\|\grad \phi^h\|
    \leq \frac{\nu}{16}\|\grad \phi^h\|^2+C(\nu)\|\grad\bm{\eta}\|^2.
\end{equation*}
\begin{equation*}
     |(p-\lambda^h,\grad\cdot\phi^h)|\leq \|p-\lambda^h\|\|\grad\cdot  \phi^h\|\leq \frac{\gamma}{4}\|\grad\cdot\phi^h\|^2+C(\gamma^{-1})\|p-\lambda^h\|^2.
\end{equation*}
\begin{equation*}
    |\gamma(\grad\cdot\eta,\grad\cdot\phi^h)|\leq\gamma\|\grad\cdot\eta\|\|\grad\cdot\phi^h\|\leq\frac{\gamma}{4}\|\grad\cdot\phi^h\|^2+C(\gamma)\|\grad\cdot\eta\|^2.
\end{equation*}
For the fifth term on the right, use the H\"{o}lder's inequality,
\begin{align*}
    (C_s\delta)^2C_2r\|\grad\bm{\eta}\|_{L^3}\|\grad \phi^h\|_{L^3}&\leq (C_s\delta)^2\left\{\frac{C_1}{3}\|\grad \phi^h\|_{L^3}^3+\frac{2}{3}C_1^{-1/2}r^{3/2}\|\grad\bm{\eta}\|_{L^3}^{3/2}\right\}.
\end{align*}
Next, for the first and the third nonlinear terms, here we follow the estimates in \cite[p.1007-1008, equations (4.5) and (4.6)]{Layton2002error} and we omit the details.
\begin{align}
    |b^*(\bm{\eta}, w, \phi^h)|
    &\leq \frac{1}{4}\|\grad w\|_{L^3}^2\| \phi^h\|^2+\frac{1}{4}\|\bm{\eta}\|_{L^6}^2
    +\epsilon_1\|\grad \phi^h\|_{L^3}^3 
    +C\epsilon_1^{-1/2}\| w\|^{3/2}\|\bm{\eta}\|_{L^6}^{3/2} . \label{nonlinear_term1}\\
    \nonumber\\
    |b^*( \phi^h, w, \phi^h)|&\leq \|\grad w\|_{L^3}\|\grad \phi^h\|_{L^3}^2,\nonumber\\ 
    &\leq \|\grad w\|_{L^3}(\| \phi^h\|^{1/2}\|\grad \phi^h\|^{1/2})^2, \nonumber\\
    &\leq \frac{\nu}{16}\|\grad \phi^h\|^2+C(\nu)\|\grad w\|_{L^3}^2\| \phi^h\|^2. \nonumber\\
    \nonumber\\
    |b^*( w^h,\bm{\eta}, \phi^h)|
    &\leq \frac{1}{4}\| w\|_{L^6}^2\| \phi^h\|^2+\frac{1}{4}\|\grad\bm{\eta}\|_{L^3}^2+\epsilon_2\|\grad \phi^h\|_{L^3}^3+C\epsilon_2^{-1/2}\| w^h\|^{3/2}\|\bm{\eta}\|_{L^6}^{3/2}. \label{nonlinear_term3}
\end{align}
Setting $\epsilon_1=\epsilon_2=\frac{1}{6}C_1(C_s\delta)^2$ and collecting all the terms gives
\begin{gather*}
    \frac{1}{2}\frac{d}{dt}\Big\{\| \phi^h\|^2+\frac{C_s^4}{\mu^2}\delta^2\|\grad \phi^h\|^2\Big\}+ \frac{\nu}{8}\|\grad \phi^h\|^2+ \frac{\gamma}{2}\|\grad\cdot\phi^h\|^2+\frac{1}{3}C_1(C_s\delta)^2\|\grad \phi^h\|_{L^3}^3 \\
    \leq \Big[C(\nu)\|\grad w\|_{L^3}^2+\frac{1}{4}\|\grad w\|_{L^3}^2+\frac{1}{4}\| w\|_{L^6}^2\Big]\| \phi^h\|^2 \\
    +\Big\{C(\nu)\Big[\|\bm{\eta}_t\|_{-1}^2+(\frac{C_s^4\delta^2}{\mu^2})^2\|\grad\bm{\eta}_t\|^2+\|\grad\bm{\eta}\|^2\Big]\\ +C(\gamma^{-1})\|p-\lambda^h\|^2+C(\gamma)\|\grad\cdot\gamma\|^2
    + (C_s\delta)^2 r^{3/2}\|\grad\bm{\eta}\|_{L^3}^{3/2}
    \\
    +\frac{1}{4}\|\bm{\eta}\|_{L^6}^2+\delta^{-1}\| w\|^{3/2}\|\bm{\eta}\|_{L^6}^{3/2}+\frac{1}{4}\|\grad\bm{\eta}\|_{L^3}^2+\delta^{-1}\| w^h\|^{3/2}\|\bm{\eta}\|_{L^6}^{3/2}\Big\}.
\end{gather*}
Denote $a(t):=C(\nu)\|\grad w\|_{L^3}^2+\frac{1}{4}\|\grad w\|_{L^3}^2+\frac{1}{4}\| w\|_{L^6}^2$ and its antiderivative is 
\begin{equation*}
    A(t):=\int_0^T a(t)\ dt<\infty\ \text{for}\ 
     w\in L^2(0,T;W^{1,3}(\Omega))\cap L^2(0,T;L^6(\Omega)).
\end{equation*}
Multiplying through by the integrating factor $e^{-A(t)}$ gives
\begin{gather*}
    \frac{d}{dt}\left[\frac{1}{2}e^{-A(t)}\left(\| \phi^h\|^2+\frac{C_s^4\delta^2}{\mu^2}\|\grad \phi^h\|^2\right)\right]\\+e^{-A(t)}\left[\frac{\nu}{8}\|\grad \phi^h\|^2+\frac{\gamma}{2}\|\grad\cdot\phi^h\|^2+\frac{1}{3}C_1(C_s\delta)^2\|\grad \phi^h\|_{L^3}^3\right] \\
    \leq e^{-A(t)}\Big\{C(\nu)\Big[\|\bm{\eta}_t\|_{-1}^2+(\frac{C_s^4\delta^2}{\mu^2})^2\|\grad\bm{\eta}_t\|^2+\|\grad\bm{\eta}\|^2\Big]\\+C(\gamma^{-1})\|p-\lambda^h\|^2+C(\gamma)\|\grad\cdot\eta\|^2 
    + (C_s\delta)^2 r^{3/2}\|\grad\bm{\eta}\|_{L^3}^{3/2}+\frac{1}{4}\|\bm{\eta}\|_{L^6}^2\\+\delta^{-1}\| w\|^{3/2}\|\bm{\eta}\|_{L^6}^{3/2}+\frac{1}{4}\|\grad\bm{\eta}\|_{L^3}^2+\delta^{-1}\| w^h\|^{3/2}\|\bm{\eta}\|_{L^6}^{3/2}\Big\}.
\end{gather*}
Integrating over $[0,T]$ and multiplying through by $e^{A(t)}$ gives
\begin{gather*}
    \frac{1}{2}\left\{\| \phi^h(T)\|^2+\frac{C_s^4\delta^2}{\mu^2}\|\grad \phi^h(T)\|^2\right\}\\+\int_0^T\Big( \frac{\nu}{8}\|\grad \phi^h\|^2+\frac{\gamma}{2}\|\grad\cdot\phi^h\|^2+\frac{1}{3}C_1(C_s\delta)^2\|\grad \phi^h\|^3_{L^3}\Big) dt \\
    \leq \exp(\int_0^T a(t) dt)\Big\{\frac{1}{2}\Big(\| \phi^h(0)\|^2+\frac{C_s^4\delta^2}{\mu^2}\|\grad \phi^h(0)\|^2\Big) \\
    +\int_0^T\Big[ C(\nu)\Big(\|\bm{\eta}_t\|_{-1}^2+(\frac{C_s^4\delta^2}{\mu^2})^2\|\grad\bm{\eta}_t\|^2+\|\grad\bm{\eta}\|^2\Big)\\+C(\gamma^{-1})\|p-\lambda^h\|^2+C(\gamma)\|\grad\cdot\eta\|^2 
    + (C_s\delta)^2 r^{3/2}\|\grad\bm{\eta}\|_{L^3}^{3/2}+\frac{1}{4}\|\bm{\eta}\|_{L^6}^2\\+\delta^{-1}\| w\|^{3/2}\|\bm{\eta}\|_{L^6}^{3/2}+\frac{1}{4}\|\grad\bm{\eta}\|_{L^3}^2+\delta^{-1}\| w^h\|^{3/2}\|\bm{\eta}\|_{L^6}^{3/2}\Big] dt\Big\}.
\end{gather*}
Apply the H\"{o}lder's inequality gives
\begin{align*}
    \int_0^Tr^{3/2}\|\grad\bm{\eta}\|_{L^3}^{3/2}\ dt\leq \left(\int_0^T r^3\ dt\right)^{1/2}\|\grad\bm{\eta}\|_{L^3(0,T;L^3)}^{3/2} ,\\
    \int_0^T\| w\|^{3/2}\|\bm{\eta}\|_{L^6}^{3/2}\ dt\leq \| w\|_{L^2(0,T;L^2)}^{3/2}\|\bm{\eta}\|_{L^6(0,T;L^6)}^{3/2} ,\\
    \int_0^T\| w^h\|^{3/2}\|\bm{\eta}\|_{L^6}^{3/2}\ dt\leq\| w^h\|_{L^2(0,T;L^2)}^{3/2}\|\bm{\eta}\|_{L^6(0,T;L^6)}^{3/2} .
\end{align*}
$\| w\|_{L^2(0,T;L^2)}$ and $\| w^h\|_{L^2(0,T;L^2)}$ are bounded by problem data by stability bound. Here $r=\max\{\|\grad w\|_{L^3},\|\grad\widetilde{ w}\|_{L^3}\}$ and $\left(\int_0^T r^3\ dt\right)^{1/2}=\|\grad w\|_{L^3(0,T;L^3)}^{3/2}$ or\\ $\|\grad\widetilde{ w}\|_{L^3(0,T;L^3)}^{3/2}$ also bounded.
Using triangle inequality: $\|\bm{e}\|\leq \| \phi^h\|+\|\bm{\eta}\|$, we obtain the desired result.
\end{proof}
\begin{remark}
If $\tilde{ w}$ is taken to be the Stokes projection, then $\|\grad\bm{\eta}\|^2$ does not occur at the RHS.
\end{remark}
\begin{remark}\label{re:nonlinear}
Considering the nonlinear terms \eqref{nonlinear_term1} and \eqref{nonlinear_term3}, alternatively we have
\begin{align*}
    |b^*(\eta,w,\phi^h)|
    &\leq M\|\grad\eta\|\|\grad w\|\|\grad\phi^h\|\leq\epsilon\|\grad\phi^h\|^2+\frac{1}{4\epsilon}M^2\|\grad w\|^2\|\grad\eta\|^2.
    \\
    |b^*(w^h,\eta,\phi^h)|
    &\leq C\|w^h\|^{1/2}\|\grad w^h\|^{1/2}\|\grad\eta\|\|\grad\phi^h\|,
    \\&\leq\epsilon\|\grad\phi^h\|^2+C(\epsilon^{-1})\|w^h\|\|\grad w^h\|\|\grad\eta\|^2.
\end{align*}
By taking $\epsilon=\nu/32$, we can avoid the term $\delta^{-1}\|\eta\|_{L^6}^{3/2}$ at the RHS but instead we have $\nu^{-1}\|\grad\eta\|^2$.
\end{remark}

\subsection{Time discretization of the Corrected Smagorinsky model}
This subsection presents the unconditionally stable, linearly implicit, full discretization of  \eqref{CSM8}. Let the time-step and other quantities be denoted by 
\begin{align*}
    \text{time-step}=k,\ t_n=nk,\ f_n( x)=f( x,t_n), \\
     w_n^h( x)=\text{approximation to}\  w( x,t_n), \\
    p_n^h( x)=\text{approximation to}\ p( x,t_n).
\end{align*}

We perform the finite element spatial discretization and the first-order Backward Euler scheme for time discretization to get the following full discretization: 
Given $( w_n^h,p_n^h)\in (X^h,Q^h)$, find $( w_{n+1}^h,p_{n+1}^h)\in (X^h,Q^h)$ satisfying

\begin{equation}\label{method}
\begin{aligned}
   \Big( \frac{ w_{n+1}^h- w_n^h}{k}, v^h\Big)+\frac{C_s^4\delta^2}{\mu^2}\Big(&\frac{\grad  w^h_{n+1}-\grad  w^h_n}{k},\grad  v^h\Big)+b^*( w^h_n, w^h_{n+1}, v^h) \\
    +\nu(\grad  w^h_{n+1},\grad  v^h)&+(C_s\delta)^2(|\grad  w^h_n|\grad  w^h_{n+1},\grad  v^h)\\
    +\gamma(\grad\cdot w_{n+1}^h,\grad\cdot v^h)-(p^h_{n+1},\div  v^h)&=(f_{n+1}( x), v^h)\ \forall\  v^h\in X^h, \\
 (\div  w^h_{n+1},q^h)&=0 \ \forall\ q^h\in Q^h  .
\end{aligned}
\end{equation}

This method is semi-implicit. 
We shall prove it is unconditionally stable in \cref{thm:be}.
\begin{theorem}\label{thm:be}
\eqref{method} is unconditionally energy stable. For any $N\geq 1$,
\begin{equation}
\begin{aligned}\label{discrete-energy}
    &\Big(\frac{1}{2}\| w^h_N\|^2+\frac{1}{2}\frac{C_s^4\delta^2}{\mu^2}\|\grad  w^h_N\|^2\Big)
    +\sum_{n=0}^{N-1}\frac{1}{2}\Big(\| w^h_{n+1}- w^h_n\|^2\\
    +\frac{C_s^4\delta^2}{\mu^2}&\|\grad  w^h_{n+1}-\grad  w^h_n\|^2\Big) 
    +k\sum_{n=0}^{N-1}\int_\Omega [\nu+(C_s\delta)^2|\grad  w^h_n|]|\grad  w^h_{n+1}|^2\ d x+\\& +\gamma\|\grad\cdot w_{n+1}^h\|^2
    =\Big(\frac{1}{2}\| w^h_0\|^2+\frac{1}{2}\frac{C_s^4\delta^2}{\mu^2}\|\grad  w^h_0\|^2\Big)+k\sum_{n=0}^{N-1} (f_{n+1}, w^h_{n+1}).
\end{aligned}
\end{equation}
\end{theorem}
\begin{proof}
 Multiply \eqref{method} by k and take  $ v^h= w^h_{n+1}$. Use Lemma \eqref{trilinear1} to get\\ $b^*( w^h_n, w^h_{n+1}, w^h_{n+1})$ $=0$. Hence,
\begin{gather*}
    \| w^h_{n+1}\|^2-( w^h_{n+1}, w^h_n)+\frac{C_s^4\delta^2}{\mu^2}\|\grad  w^h_{n+1}\|^2-\frac{C_s^4\delta^2}{\mu^2}(\grad  w^h_{n+1},\grad  w^h_n) \\
    +\gamma\|\grad\cdot w_{n+1}^h\|^2 +k \int_{\Omega} [\nu+(C_s\delta)^2|\grad  w^h_n|]|\grad  w^h_{n+1}|^2\ d x= k(f_{n+1}, w^h_{n+1}).
\end{gather*}
For the second and fourth term, apply the polarization identity \eqref{polar},
\begin{gather*}
    ( w^h_{n+1}, w^h_n)=\frac{1}{2}\| w^h_{n+1}\|^2+\frac{1}{2}\| w^h_n\|^2-\frac{1}{2}\| w^h_{n+1}- w^h_n\|^2, \\
    (\grad  w^h_{n+1},\grad  w^h_n)=\frac{1}{2}\|\grad  w^h_{n+1}\|^2+\frac{1}{2}\|\grad  w^h_n\|^2-\frac{1}{2}\|\grad  w^h_{n+1}-\grad  w^h_n\|^2.
\end{gather*}
Collecting terms and summing from $n=0$ to $N-1$, we get the result.
\end{proof} 
\begin{remark}
\eqref{discrete-energy} is an energy equality, we can identify the following quantities:
\begin{enumerate}
    \item Model kinetic energy $=\frac{1}{2}\| w^h_N\|^2+\frac{1}{2}\frac{C_s^4\delta^2}{\mu^2}\|\grad  w^h_N\|^2$.
    \item Eddy viscosity dissipation $=\int_\Omega (C_s\delta)^2|\grad  w^h_n||\grad  w^h_{n+1}|^2\ d x$.
    \item Numerical diffusion $=\frac{1}{2}(\| w^h_{n+1}- w^h_n\|^2+\frac{C_s^4\delta^2}{\mu^2}\|\grad  w^h_{n+1}-\grad  w^h_n\|^2)$. This numerical diffusion arises due to the Backward Euler scheme.
\end{enumerate}
\end{remark}
\begin{remark}
The energy equality \eqref{discrete-energy} can be also written as
\begin{equation*}
\begin{aligned}
    \frac{1}{2k}(\| w^h_{n+1}\|^2-\| w^h_n\|^2)+\frac{1}{2k}\| w^h_{n+1}- w^h_n\|^2+\nu\|\grad  w^h_{n+1}\|^2 +\gamma\|\grad\cdot w_{n+1}^h\|^2 \\
    + \Bigg\{\frac{C_s^4\delta^2}{2k\mu^2}(\|\grad  w^h_{n+1}\|^2-\|\grad  w^h_n\|^2)+\frac{C_s^4\delta^2}{2k\mu^2}\|\grad  w^h_{n+1}-\grad  w^h_n\|^2\\
    +\int_\Omega (C_s\delta)^2|\grad  w^h_n||\grad  w^h_{n+1}|^2\ d x \Bigg\} 
    = (f_{n+1}, w^h_{n+1}).
\end{aligned}
\end{equation*}
Line one and the RHS are from the backward Euler discretization of usual NSE. The bracketed term is a discretized form of model dissipation at $t=t_{n+1}$. Here the term model dissipation in the paper can be positive or negative. When it is positive, it aggregates energy from mean to fluctuations. And when it is negative, energy is being transferred from fluctuations back to mean.
\end{remark}
\begin{remark}
For \eqref{method}, the model dissipation is
\begin{gather*}
    \text{MD}^{n+1}=\frac{C_s^4\delta^2}{2k\mu^2}(\|\grad  w^h_{n+1}\|^2-\|\grad  w^h_n\|^2)+\frac{C_s^4\delta^2}{2k\mu^2}\|\grad  w^h_{n+1}-\grad  w^h_n\|^2\\
    +\int_\Omega (C_s\delta)^2|\grad  w^h_n||\grad  w^h_{n+1}|^2\ d x.
\end{gather*}
\end{remark}

In this Test 8.2, we test use both Backward Euler and Crank-Nicolson to see the difference. 
We perform the finite element spatial discretization and the linearly implicit Crank-Nicolson (also called CNLE-CN with Linear Extrapolation) scheme for time discretization to get the following full discretization: for function w, we denote
\begin{equation*}
     w^h_{n+\frac{1}{2}}=\frac{ w^h_n+ w^h_{n+1}}{2},\quad \Tilde{ w}^h_{n+\frac{1}{2}}=\frac{3 w^h_n- w^h_{n-1}}{2}.
\end{equation*}
Given $( w_n^h,p_n^h)\in (X^h,Q^h)$, find $( w_{n+1}^h,p_{n+1}^h)\in (X^h,Q^h)$ satisfying

\begin{equation}\label{cn-method}
\begin{aligned}
   \Big( \frac{ w_{n+1}^h- w_n^h}{k}, v^h\Big)+\frac{C_s^4\delta^2}{\mu^2}\Big(&\frac{\grad  w^h_{n+1}-\grad  w^h_n}{k},\grad  v^h\Big)+b^*(\Tilde{ w}^h_{n+\frac{1}{2}}, w^h_{n+\frac{1}{2}}, v^h) \\
    +\nu(\grad  w^h_{n+\frac{1}{2}},\grad  v^h)&+(C_s\delta)^2(|\grad \Tilde{ w}^h_{n+\frac{1}{2}}|\grad  w^h_{n+\frac{1}{2}},\grad  v^h)\\
    +\gamma(\grad\cdot w_{n+\frac{1}{2}}^h,\grad\cdot v^h)-(p^h_{n+\frac{1}{2}},\div  v^h)&=(f_{n+\frac{1}{2}}( x), v^h)\ \forall\  v^h\in X^h, \\
 (\div  w^h_{n+\frac{1}{2}},q^h)&=0 \ \forall\ q^h\in Q^h  .
\end{aligned}
\end{equation}
We will prove it is unconditionally stable in \cref{thm:stable-cn}.
\begin{theorem}\label{thm:stable-cn}
\eqref{cn-method} is unconditionally energy stable. For any $N\geq 1$,
\begin{equation}\label{cn-discrete-energy}
\begin{aligned}
    &\left(\frac{1}{2}\| w^h_{N}\|^2+\frac{1}{2}\frac{C_s^4\delta^2}{\mu^2}\|\grad  w^h_N\|^2\right)+k\sum_{n=0}^{N-1}\int_\Omega[\nu+(C_s\delta)^2|\grad\Tilde{ w}^h_{n+\frac{1}{2}}|]|\grad  w^h_{n+\frac{1}{2}}|^2\ d x \\& +\gamma\|\grad\cdot w_{n+\frac{1}{2}}^h\|^2 = \left(\frac{1}{2}\| w^h_0\|^2+\frac{1}{2}\frac{C_s^4\delta^2}{\mu^2}\|\grad  w^h_0\|^2\right)+k\sum_{n=0}^{N-1}\left(f_{n+\frac{1}{2}}, w^h_{n+\frac{1}{2}}\right).
\end{aligned}
\end{equation}
\end{theorem}
\begin{proof}
Multiply \eqref{cn-method} by k and take $ v^h= w^h_{n+\frac{1}{2}}$. Use Lemma \eqref{trilinear1} to get $b^*(\Tilde{ w}^h_{n+\frac{1}{2}},  w^h_{n+\frac{1}{2}}, w^h_{n+\frac{1}{2}})=0$. Hence,
\begin{align*}
    \frac{1}{2}\| w^h_{n+1}\|^2-\frac{1}{2}\| w^h_n\|^2+\frac{1}{2}\frac{C_s^4\delta^2}{\mu^2}\|\grad  w^h_{n+1}\|^2-\frac{1}{2}\frac{C_s^4\delta^2}{\mu^2}\|\grad  w^h_n\|^2 \\  +\gamma\|\grad\cdot w_{n+\frac{1}{2}}^h\|^2
    +k\int_\Omega[\nu+(C_s\delta)^2|\grad\Tilde{ w}^h_{n+\frac{1}{2}}|]|\grad  w^h_{n+\frac{1}{2}}|^2\ d x=k(f_{n+\frac{1}{2}}, w^h_{n+\frac{1}{2}}).
\end{align*}
Collecting terms and summing from $n=0$ to $N-1$, we get the result.
\end{proof}
\begin{remark}
\eqref{cn-discrete-energy} is an energy equality, we can identify the following quantities:
\begin{enumerate}
    \item Model kinetic energy $=\frac{1}{2}\| w^h_{N}\|^2+\frac{1}{2}\frac{C_s^4\delta^2}{\mu^2}\|\grad  w^h_N\|^2$.
    \item Eddy viscosity dissipation $=\int_\Omega(C_s\delta)^2|\grad\Tilde{ w}^h_{n+\frac{1}{2}}||\grad  w^h_{n+\frac{1}{2}}|^2\ d x$.
    \item No Numerical diffusion.
\end{enumerate}
\end{remark}
\begin{remark}
The energy equality can be also written as
\begin{gather*}
    \frac{1}{2k}(\| w^h_{n+1}\|^2-\| w^h_n\|^2)+\nu\|\grad  w^h_{n+\frac{1}{2}}\|^2  +\gamma\|\grad\cdot w_{n+\frac{1}{2}}^h\|^2
    \\
    +\Bigg\{\frac{C_s^4\delta^2}{2k\mu^2}(\|\grad  w^h_{n+1}\|^2-\|\grad  w^h_n\|^2)+\int_\Omega (C_s\delta)^2|\grad\Tilde{ w}^h_{n+\frac{1}{2}}||\grad  w^h_{n+\frac{1}{2}}|^2\ d x\Bigg\} \\
    =(f_{n+\frac{1}{2}}, w^h_{n+\frac{1}{2}}).
\end{gather*}
Line one and line three are from the CNLE discretization of usual NSE. The bracketed term in the second line is a discretized form of model dissipation at $t=t_{n+1}$.
\end{remark}
\begin{remark}
For \eqref{cn-method}, the model dissipation is
\begin{equation*}
    MD^{n+1}=\frac{C_s^4\delta^2}{2k\mu^2}(\|\grad  w^h_{n+1}\|^2-\|\grad  w^h_n\|^2)+\int_\Omega (C_s\delta)^2|\grad\Tilde{ w}^h_{n+\frac{1}{2}}||\grad  w^h_{n+\frac{1}{2}}|^2\ d x.
\end{equation*}
\end{remark}

\section{Numerical Tests}
In this section, we perform two numerical tests.
In the first test, we show the numerical error and the rate of convergence of the Backward Euler scheme. In the second test, we show among Backward Euler (BE) and Crank-Nicolson with Linear Extrapolation (CNLE), CNLE exhibits intermittent backscatter.
\subsection{A test with exact solution} (Taken from V. DeCaria, W. J. Layton and M. McLaughlin \cite{Victor2017artificial})
The first experiment tests the accuracy of the Corrected Smagorinsky Model \eqref{CSM8} and convergence rate of \eqref{method}. 
The following test has an exact solution for the 2D Navier Stokes problem. \\
Let the domain $\Omega=(-1,1)\times(-1,1)$. The exact solution is as follows:
\begin{align*}
    & u(x,y,t)=\pi\sin t(\sin 2\pi y\sin ^2 \pi x,-\sin 2\pi x\sin^2\pi y).\\
    &p(x,y,t)=\sin t\cos \pi x\sin\pi y.
\end{align*}
This is inserted into the CSM and the body force $f(x,t)$ is calculated. \\
Uniform meshes were used with 270 nodes per side on the boundary and the degrees of freedom for the velocity space is 292681 and for the pressure space is 73441. The mesh is fine enough compared to the time-step so that the main error from time-steps is only considered here. Taylor-Hood elements (P2-P1) were used in this test. We ran the test up to $T=10.$ We take $C_s=0.1, \ \mu=0.4,\ \delta$ is taken to be the shortest edge of all triangles. The norms used in the table heading are defined as follows,
$$\| w\|_{\infty,0}:= \text{ess} \sup_{0<t<T} \| w\|_{L^2(\Og)}\ \text{and}\ \| w\|_{0,0}:=\left(\int_0^T\| w(\cdot, t)\|_{L^2(\Og)}^2\ dt\right)^{1/2}.$$
\begin{table}[H]
    \centering
    \begin{tabular}{||c|c|c|c|c|c|c||}
    \hline
         $\Dt t$ & $\| w- w^h\|_{\infty,0}$ &rate& $\|\grad( w- w^h)\|_{0,0}$ &rate& $\|p-p^h\|_{0,0}$ &rate \\
         \hline
         0.05&3.27068&-&5.25129&-&0.640537&- \\
         \hline
         0.02&0.823036&1.506 &1.59313&1.302 &0.235862&1.091 
 \\
         \hline
         0.01&0.348629&1.239  &0.739145&1.108  &0.108216&1.124 
 \\
         \hline
         0.005&0.169429& 1.041   &0.39714&0.89621&0.0470406&1.202  \\
         \hline
    \end{tabular}
    \caption{Numerical error and temporal convergence rate, $Re=5,000,\ T_{final}=10, \ C_s=0.1, \ \mu=0.4, \ \delta=0.0104757$.}
    \label{tab:test1a-Re5000}
\end{table}
From the \cref{tab:test1a-Re5000}, we see the temporal convergence rate is 1 which is expected from Backward Euler \eqref{method} discretization.
\begin{table}[H]
    \centering
    \begin{tabular}{||c|c|c|c|c|c|c||}
    \hline
         $h=\delta$ & $\| w- w^h\|_{\infty,0}$ &rate& $\|\grad( w- w^h)\|_{0,0}$ &rate& $\|p-p^h\|_{0,0}$ &rate \\
         \hline
         0.08571&57.9769&-&88.1677&-&14.5602&- \\
         \hline
         0.04221&1.41386&5.244  &3.30974&4.635   &0.313994&5.418 
 \\
         \hline
         0.02095&0.407421&1.776   &0.95483&1.774  &0.0562327&2.455
 \\
         \hline
         0.01048&0.169429&  1.266 &0.39714&1.266&0.0470406&0.258    \\
         \hline
    \end{tabular}
    \caption{Numerical error and spatial convergence rate, $Re=5,000, \ T_{final}=10, \ C_s=0.1, \ \mu=0.4, \ \Delta t=0.005$.}
    \label{tab:test1b-Re5000}
\end{table}
Using Taylor-Hood elements, \cref{thm:numerial-error} predicts a convergence rate in space of $O(h^{1.75})$, with a moderate constant,  for $\| w- w^h\|_{\infty,0}$ and $\|\grad( w- w^h)\|_{0,0}$. But with the estimates in \cref{re:nonlinear}, the order of convergence is $O(h^2)$, with a large constant $\frac{1}{\nu}$. In \cref{tab:test1b-Re5000}, third and fifth  column show rates $O(h^{1.78})$ until the error plateaus (last line) at the error in the time discretization (last line in \cref{tab:test1a-Re5000}). There is still some gap between the theoretical convergence rate and the actual convergence rate we get in \cref{tab:test1b-Re5000}.
The behavior of the pressure error for this test problem is unclear as well in \cref{tab:test1b-Re5000}.

\subsection{Test2. Flow between offset cylinder} (Taken from N. Jiang and W. J. Layton \cite{jiang2016ev}).\label{sec:6.2} This flow problem is tested to show the transfer of energy from fluctuations back to means in the turbulent flow using the Corrected Smagorinksy Model \eqref{CSM8}. 

The domain is a disk with a smaller off center obstacle inside. Let $r_1=1,r_2=0.1,c=(c_1,c_2)=(1/2,0),$ then the domain is given by
\beas
\Omega=\{(x,y):x^2+y^2< r_1^2 \;\text{and}\; (x-c_1)^2+(y-c_2)^2> r_2^2\}.
\eeas
The flow is driven by a counterclockwise rotational body force
\beas
f(x,y,t)=(-4y*(1-x^2-y^2),4x*(1-x^2-y^2))^T,
\eeas
with no-slip boundary conditions on both circles. We discretize in space using Taylor-Hood elements. There are 80 mesh points around the outer circle and 60 mesh points around the inner circle. The flow is driven by a counterclockwise force (f=0 on the outer circle). Thus, the flow rotates about the origin and interacts with the immersed circle. 

We start the initial condition by solving the Stokes problem. We compute up to final time $T_{final}=3$. Take $C_s=0.1, \mu=0.3, \delta$ is taken to be the shortest edge of all triangles $\approx 0.0112927$, Re=10,000.
For Backward Euler \eqref{method}, we compute the following quantities:
\begin{align*}
    \text{Model dissipation}\  MD&=\int_\Omega\Big(\frac{C_s^4\delta^2}{\mu^2}\frac{\grad  w^h_{n+1}-\grad  w^h_n}{k}\cdot\grad  w^h_{n+1}\\
    &+ (C_s\delta)^2|\grad  w^h_n||\grad  w^h_{n+1}|^2\Big)\,d x.\\
    \text{Effect of new term from CSM,}\ CSMD&=\int_\Omega \frac{C_s^4\delta^2}{\mu^2}\frac{\grad  w^h_{n+1}-\grad  w^h_n}{k}\cdot\grad  w^h_{n+1}\,d x. \\
    \text{Eddy viscosity dissipation}\ EVD&=\int_\Omega (C_s\delta)^2|\grad  w^h_n||\grad  w^h_{n+1}|^2\,d x. \\
    \text{Viscous dissipation}\ VD&=\nu\|\grad  w^h_{n+1}\|^2.
\end{align*}
For Crank-Nicolson CNLE \eqref{cn-method}, we compute the following quantities:
\begin{align*}
    \text{Model dissipation}\  MD&=\int_\Omega\Big(\frac{C_s^4\delta^2}{\mu^2}\frac{\grad  w^h_{n+1}-\grad  w^h_n}{k}\cdot\grad  w^h_{n+\frac{1}{2}}\\
    &+ (C_s\delta)^2|\grad \Tilde{ w}^h_{n+\frac{1}{2}}||\grad  w^h_{n+\frac{1}{2}}|^2\Big)\,d x. \\
    \text{Effect of new term from CSM,}\ CSMD&=\int_\Omega \frac{C_s^4\delta^2}{\mu^2}\frac{\grad  w^h_{n+1}-\grad  w^h_n}{k}\cdot\grad  w^h_{n+\frac{1}{2}}\,d x. \\
    \text{Eddy viscosity dissipation}\ EVD&=\int_\Omega (C_s\delta)^2|\grad \Tilde{ w}^h_{n+\frac{1}{2}}||\grad  w^h_{n+\frac{1}{2}}|^2\,d x. \\
    \text{Viscous dissipation}\ VD&=\nu\|\grad  w^h_{n+\frac{1}{2}}\|^2.
\end{align*}
\begin{figure}[H]
    \centering
    \includegraphics[width=\textwidth]{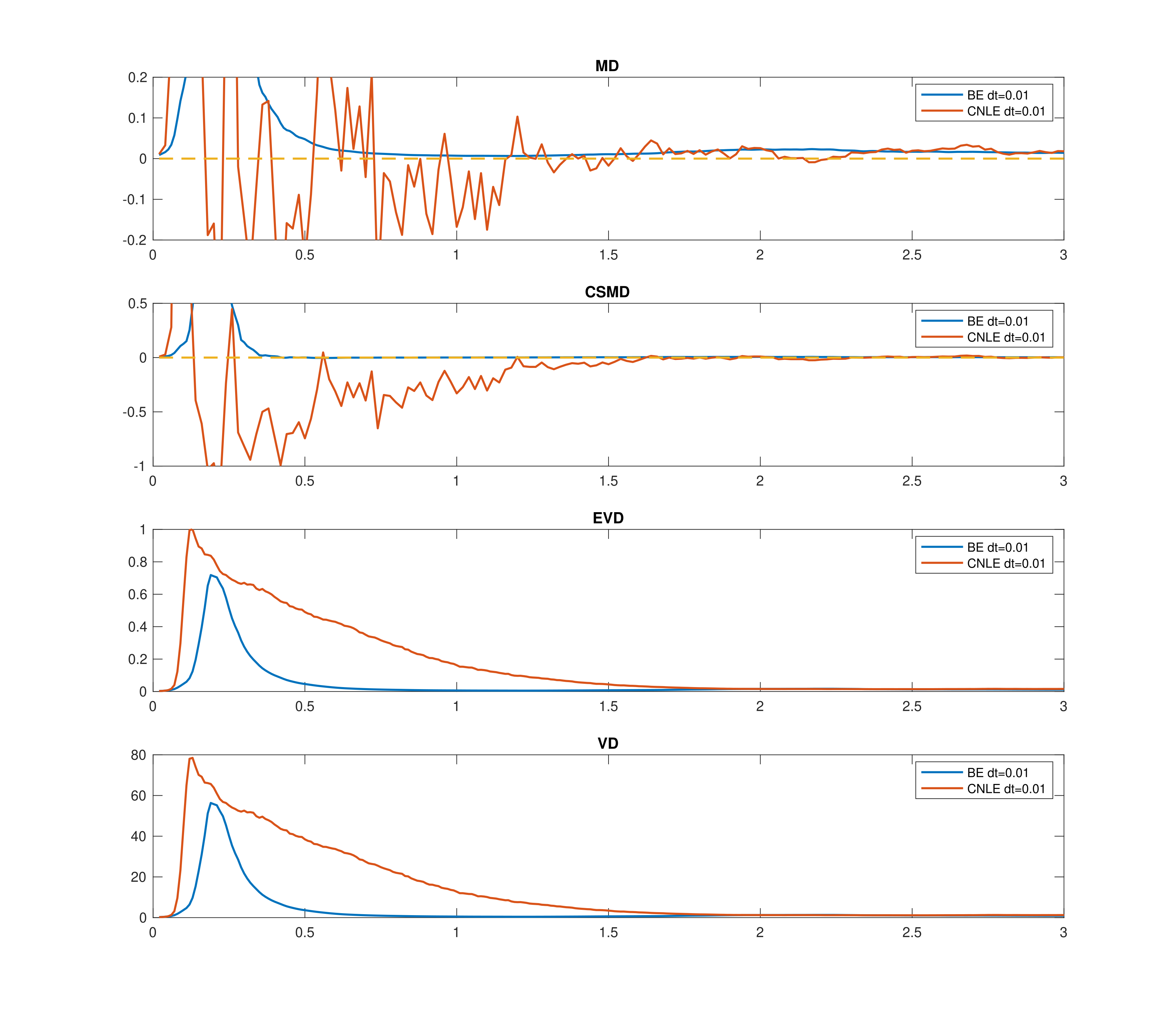}
    \caption{Comparison of Backward Euler \eqref{method} and linearized Crank-Nicolson \eqref{cn-method} with $\Dt t=0.01,\ Re=10,000, \ T_{final}=3, \ C_s=0.1, \ \mu=0.4, \ \delta=0.0112927$.}
    \label{fig:be-cn}
\end{figure}
It can be seen from the \cref{fig:be-cn}, model dissipation MD becomes negative sometimes for linearized Crank-Nicolson \eqref{cn-method} and MD are all positive for Backward Euler \eqref{method}. Only CNLE for the Corrected Smagorinksy has backscatter, which is consistent with the purpose of this model. Backward Euler has too much numerical diffusion, which makes it harder to see the backscatter from BE. 
\begin{figure}[h]
\subfloat[t=1]{\includegraphics[width=6cm]{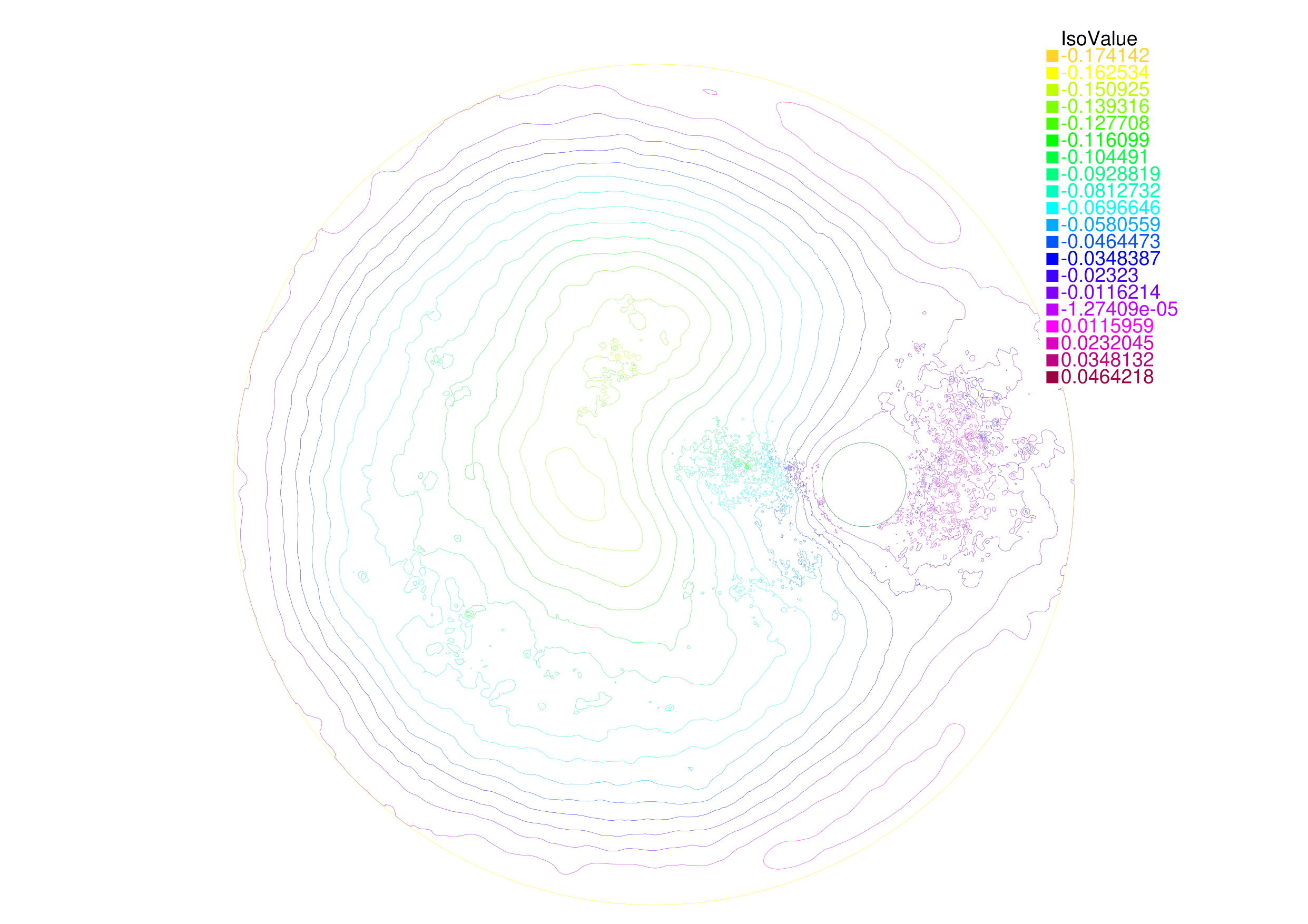}}
\subfloat[t=2]{\includegraphics[width=6cm]{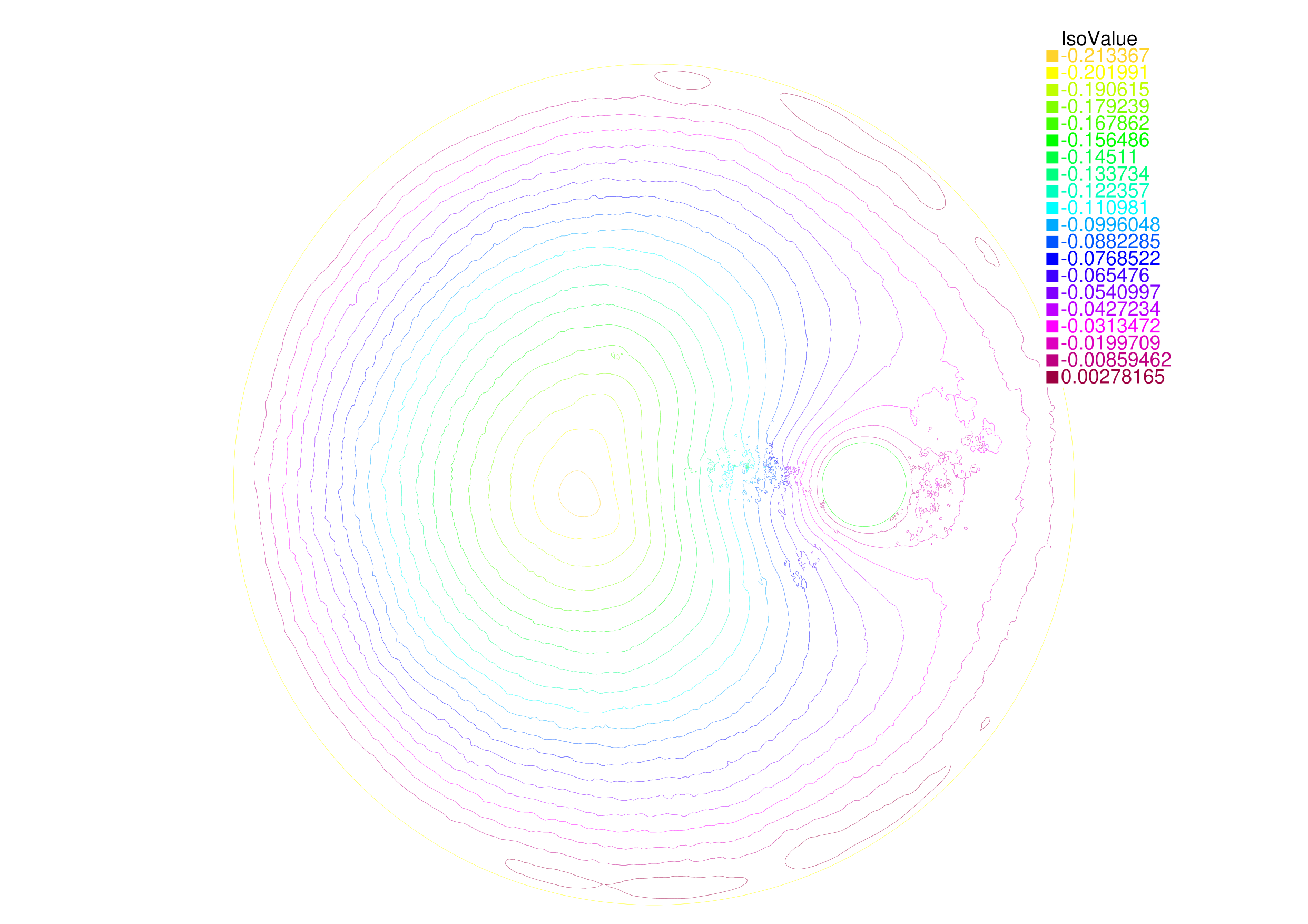}}\\
\subfloat[t=3]{\includegraphics[width=6cm]{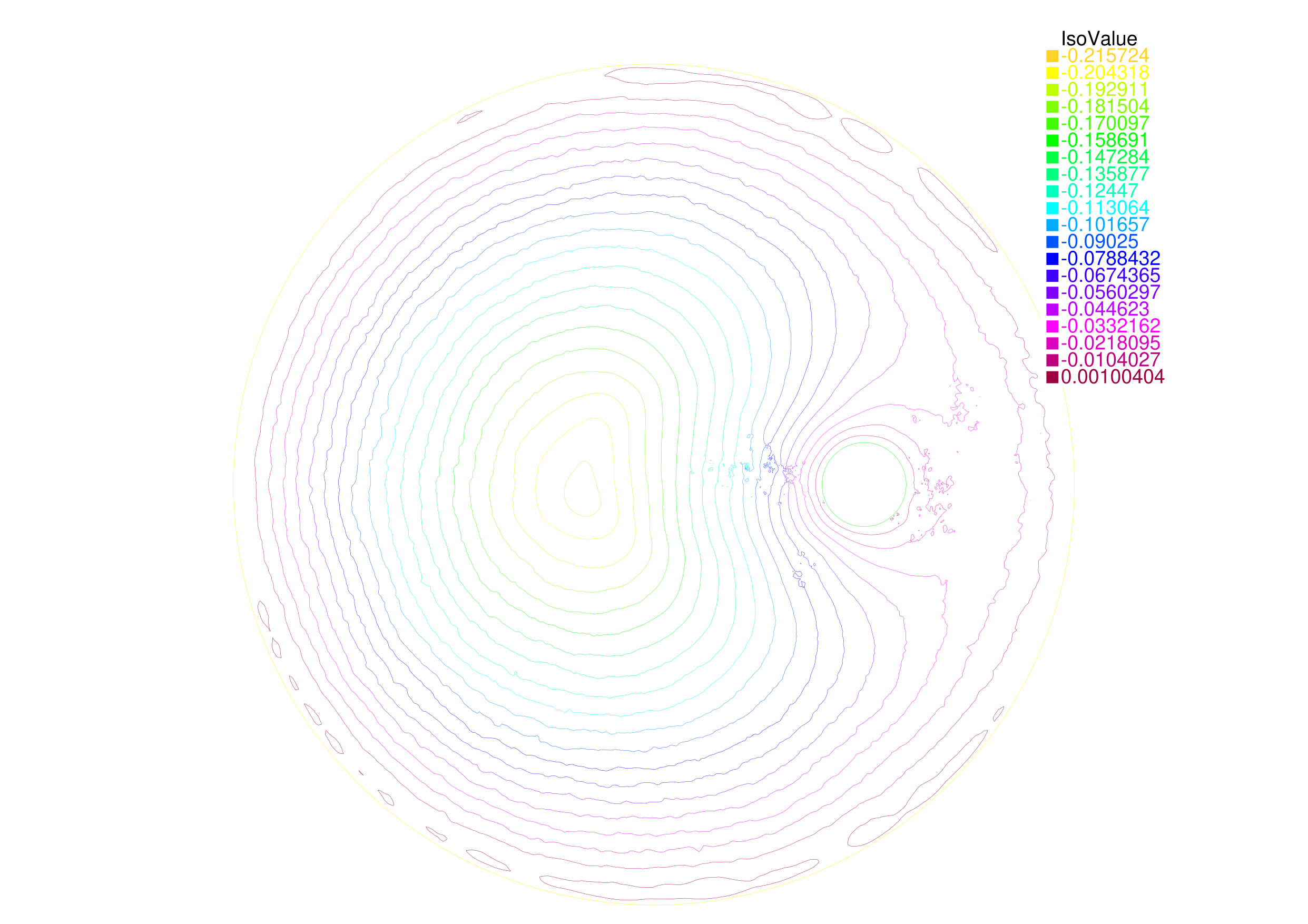}}
\subfloat[t=10]{\includegraphics[width=6cm]{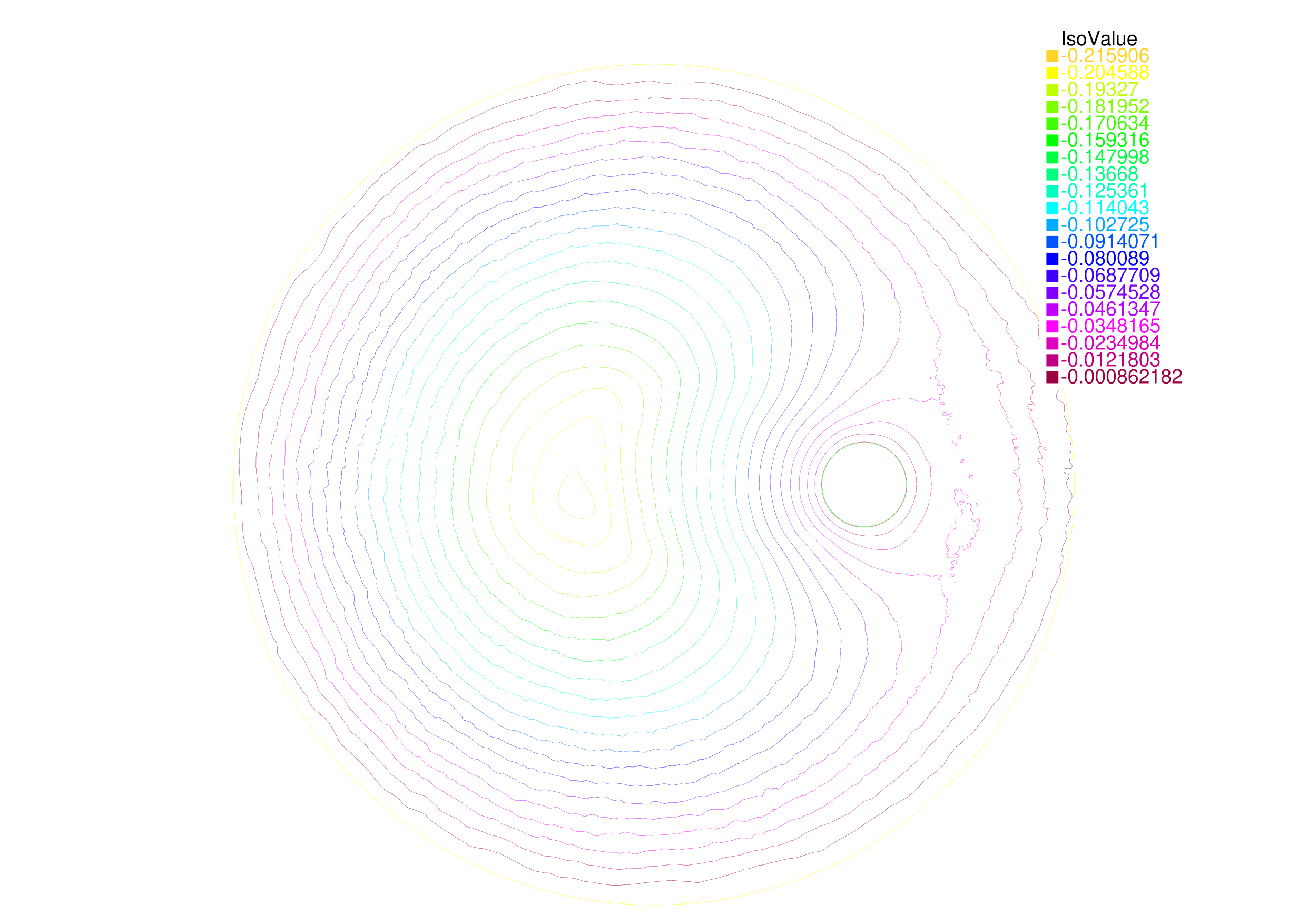}}
\caption{Streamline plot using CNLE \eqref{cn-method}. There are 270 mesh points around the outer circle and 180 mesh points around the inner circle.}
\label{fig:streamline}
\end{figure}

In the \cref{fig:streamline}, we notice the flow becomes smoother as it approaches statistical equilibrium. 

\subsubsection{Comparison with NSE and stardard Smagorinksy}
Here we compare the CSM \eqref{CSM8} with Navier Stokes \eqref{nse1} and the standard Smagorinsky \eqref{smag}. The Taylor microscale \cite{cfdbook} is defined as $\lambda_T:=\|u\|/\|\nabla u\|$,  which represents an average length scale for the flow. We use the same setting but with $Re=100,000$ to compare the Taylor microscale of each model. All numerical tests are calculated using Crank-Nicolson with grad-div stabilization $\gamma=1$.
\begin{figure}[H]
    \centering
    \includegraphics[width=\textwidth]{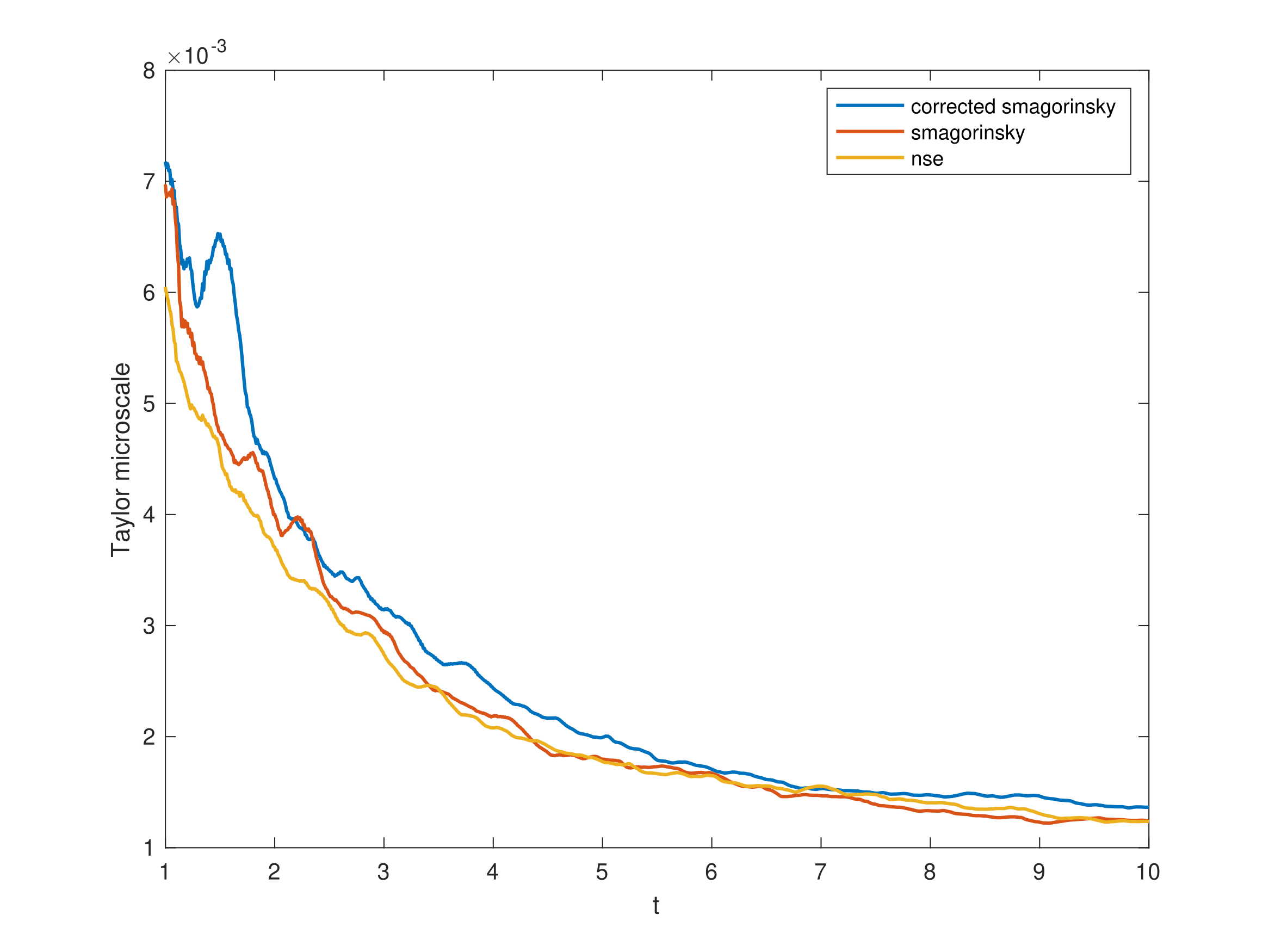}
    \caption{Taylor microscale comparison between CSM, NSE and the stardard Smagorinsky with $\Dt t=0.01,\ Re=100,000,\ T_{final}=10,\ C_s=0.1,\ \mu=0.4,\ \delta=0.0112927$.}
    \label{fig:test2_3compare}
\end{figure}
To further see the difference between these three models, here we focus on time-interval $[7,10]$ and see the relative length-scale $\lambda_T/h$ with $h$ being the meshsize.
\begin{figure}[H]
    \centering
    \includegraphics[width=\textwidth]{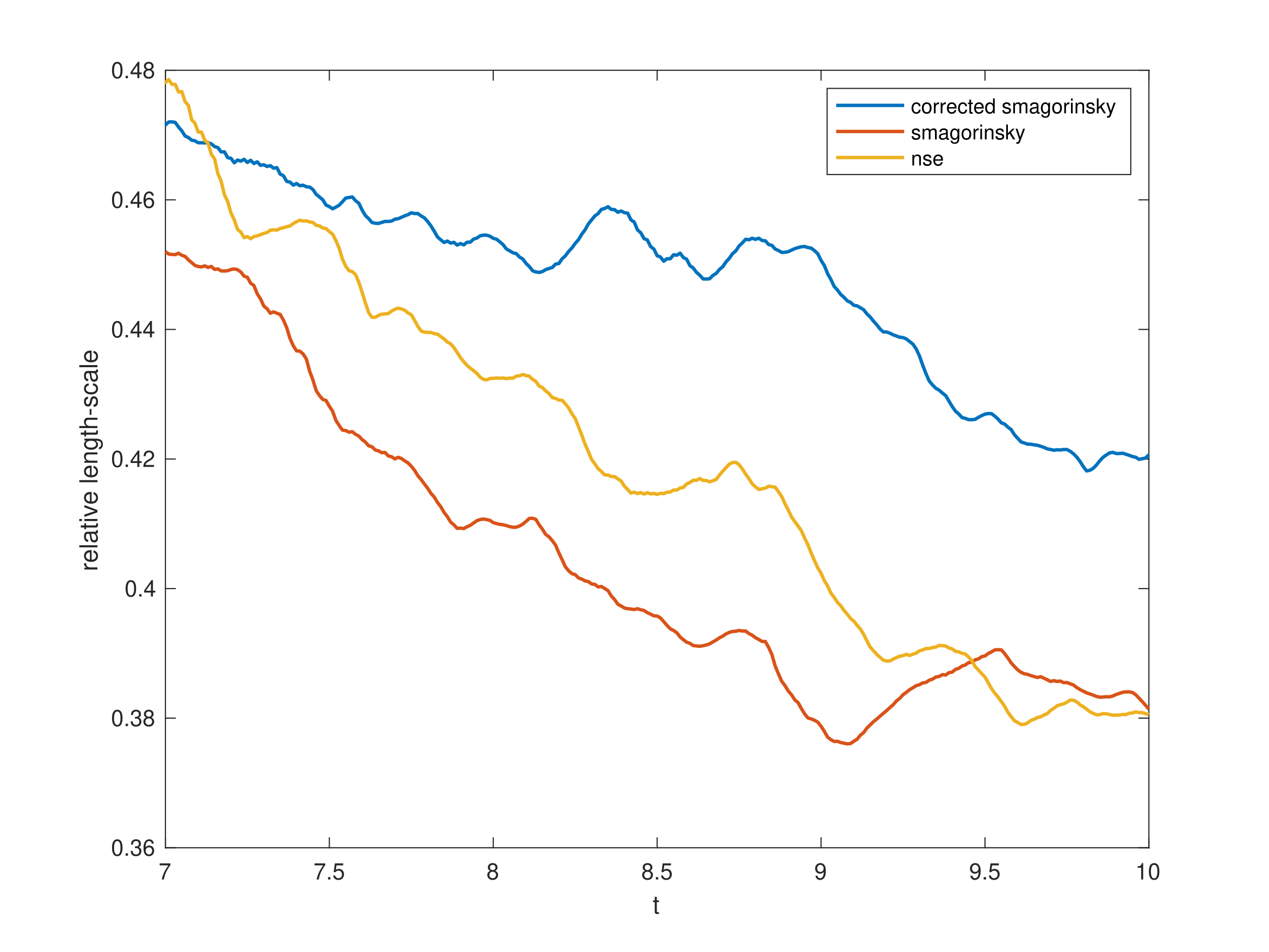}
    \caption{Relative length-scale ($\lambda_T/h$) comparison between CSM, NSE and standard Smagorinsky with $\Dt t=0.01,\  Re=100,000, \ C_s=0.1,\ \mu=0.4,\ \delta=0.0112927$, time-interval shown as $[7,10]$.}
    \label{fig:test2_3compare_zoom}
\end{figure}
From \cref{fig:test2_3compare}, notice the CSM has larger Taylor microscale. Since the CSM models backscatter, more energy is expected in velocity means. Consistent with this, the averaged length-scale of CSM is larger than Smagorinsky and NSE. And from \cref{fig:test2_3compare_zoom}, the relative length-scale of the CSM at the final time is almost twice as large as the relative length-scale calculated with NSE and standard Smagorinsky.

\section{Conclusion and future prospects}
It was demonstrated that the 
Smagor-insky Model could be extended to non-equilibrium turbulence. In addition to that, we were able to show statistical backscatter without using negative turbulent viscosities. The stability of the model, uniqueness of the model's solution, modeling error, and numerical error were analyzed in the paper. Since BE has numerical diffusion while CNLE does not, we can clearly observe backscatter from CNLE in the second numerical test.

In the next step, we can incorporate the penalty method with this model to get the desired result more efficiently. 
\section*{Acknowledgement}
We would like to thank our advisor Professor William J. Layton, for his insightful idea for developing the model and guidance throughout the research. We also humbly acknowledge Dr. Nan Jiang (Assistant Professor, Department of Mathematics, University of Florida) for providing us the FreeFem code of her paper \cite{jiang2016ev}.

\bibliographystyle{siamplain}

\bibliography{references}
\end{document}